\documentclass[12pt]{article}
\usepackage[utf8]{inputenc}
\usepackage{mathtools}
\usepackage{commath}
\newcommand*\diff{\mathop{}\!\mathrm{d}}

\usepackage{amssymb}
\usepackage{amsthm}
\usepackage{amsmath}
\usepackage{mathrsfs}
\usepackage[english]{babel}
\usepackage[total={6.5in, 9in}]{geometry}
 \usepackage[hidelinks]{hyperref} 
\usepackage[backend=biber,style=numeric,maxnames=50]{biblatex}
\usepackage{csquotes}
\usepackage{scalerel}[2014/03/10]
\usepackage[usestackEOL]{stackengine}
\usepackage{enumitem}
\usepackage{bbm}

\numberwithin{equation}{section}

\numberwithin{equation}{section}

\def\dashint{\,\ThisStyle{\ensurestackMath{%
  \stackinset{c}{.2\LMpt}{c}{.5\LMpt}{\SavedStyle-}{\SavedStyle\phantom{\int}}}%
  \setbox0=\hbox{$\SavedStyle\int\,$}\kern-\wd0}\int}
\def\ddashint{\,\ThisStyle{\ensurestackMath{%
  \stackinset{c}{.2\LMpt}{c}{.5\LMpt+.2\LMex}{\SavedStyle-}{%
    \stackinset{c}{.2\LMpt}{c}{.5\LMpt-.2\LMex}{\SavedStyle-}{%
      \SavedStyle\phantom{\int}}}}\setbox0=\hbox{$\SavedStyle\int\,$}\kern-\wd0}\int}

\makeatletter
\newcommand*{\isomorphism}{%
  \mathrel{%
    \mathpalette\@isomorphism{}%
  }%
}
\newcommand*{\@isomorphism}[2]{%
  \sbox0{$#1\simeq$}%
  \sbox2{$#1\sim$}%
  \dimen@=\ht0 %
  \advance\dimen@ by -\ht2 %
  \sbox0{%
    \lower1.9\dimen@\hbox{%
      $\m@th#1\relbar\isomorphism@joinrel\relbar$%
    }%
  }%
  \rlap{%
    \hbox to \wd0{%
      \hfill\raise\dimen@\hbox{$\m@th#1\sim$}\hfill
    }%
  }%
  \copy0 %
}
\newcommand*{\isomorphism@joinrel}{%
  \mathrel{%
    \mkern-3.4mu %
    \mkern-1mu %
    \nonscript\mkern1mu %
  }%
}
\newcommand{\Addresses}{{
  \bigskip
  \footnotesize

  \textsc{Department of Mathematics, Johns Hopkins University,
    3400 N. North Charles, Baltimore, MD 21218}\par\nopagebreak
  \textit{E-mail address}: \texttt{rtang18@jhu.edu}
}}

\makeatother

\newtheorem{theorem}{Theorem}
\newtheorem{lemma}{Lemma}
\newtheorem{definition}{Definition}
\newtheorem{proposition}{Proposition}
\newtheorem{remark}{Remark}

\title{Nonlocal Filtration Equations with Rough Kernels in the Heisenberg Group}
\author{Rong Tang}
\date{}

\begin{document}

\maketitle

\begin{abstract}
    Motivated by the extensive investigations of integro-differential equations on $\mathbb{R}^n$, we  consider nonlocal filtration type equations with rough kernels on the Heisenberg group $\mathbb{H}^n$. We prove the existence and uniqueness of weak solutions corresponding to suitable initial data. Furthermore, we obtain the large time behavior of solutions and the uniform Hölder regularity of sign-changing solutions for the porous medium type equations ($m\geq 1$). Notice that both conformal fractional operators $\mathscr{L}_{\alpha/2}$ and pure power fractional operators $\mathscr{L}^{\alpha/2}$ on the Heisenberg group $\mathbb{H}^n$ have their integral representations with suitable kernels. Therefore, all the results in this paper will hold for these equations with operators $\mathscr{L}_{\alpha/2}$ or $\mathscr{L}^{\alpha/2}$. 
\end{abstract}

\section{Introduction}
The classic filtration equations are the nonlinear parabolic equations
\begin{equation*}
    \begin{aligned}
        \partial_t u = \Delta_{\mathbb{R}^n}(u^m), ~~m>0.
    \end{aligned}
\end{equation*}
Here, $u^m = |u|^{m-1}u$. The above equations with $m > 1$ are usually called the porous medium equations. Otherwise, they are called the fast diffusion equations. There are extensive investigations of the above equations, especially in the applications of physics, see \autocite{tpm_vj}. 

In 2012, Pablo, Quirós, Rodríguez and Vázquez\autocites{afp_dpa}{agf_dpa} developed a theory of the following fractional version of filtration type equation, which can be considered as a model in statistical mechanics. 

\begin{equation*}
    \left\{
            \begin{array}{ll}
                \partial_t u+ (-\Delta)_{\mathbb{R}^n}^{\alpha/2}(|u|^{m-1}u) = 0, & x\in \mathbb{R}^n,  t>0  \\
                u(x,0) = u_0(x), & x \in \mathbb{R}^n. 
            \end{array}
    \right.
\end{equation*}
Here $0 < \alpha < 2$. The authors developed a theory of existence, uniqueness and qualitative properties of the weak solutions through the Caffarelli-Silvestre extension technique \autocite{aep_cl}. Moreover, this nonlocal equation forces the infinite propagation speed. This result is different comparing to the finite propagation speed with free boundary in the local filtration equations. This difference results from the nonlocal virtue of the fractional Laplacian operator $(-\Delta)_{\mathbb{R}^n}^{\alpha/2}$. 

Considering the Laplace-Beltrami operator $(-\Delta_M)$ on the general Riemannian manifolds $M$, the porous medium type equations have been studied for instance in \autocite{tpm_gg}. Furthermore, the pure power fractional Laplacian $(-\Delta_M)^{\alpha/2}$ can be defined through a semigroup approach. The corresponding fractional porous medium equations are investigated, see \autocite{tfp_be}.

However, the pure power fractional operator $(-\Delta_M)^{\alpha/2}$ is not conformally covariant. Let $(X, h^+)$ be a Poincar\'e-Einstein manifold with conformal infinity $(M, [g])$. The conformal fractional Laplacian $P_{\alpha/2}^g$ 
on $(M, g)$ is constructed as the Dirichlet-to-Neumann operator for a generalized eigenvalue problem on $(X, h^+)$ (see the details in \autocites{fli_cs}).

$P_{\alpha/2}^g$ is a self-adjoint pseudo-differential operator on $(M, g)$ with principal symbol the same as $(-\Delta_M)^{\alpha/2}$, and satisfies the conformal transformation relation,
\begin{equation*}
    \begin{aligned}
        P_{\alpha/2}^{\rho^{\frac{4}{n-\alpha}}g}(\phi)=\rho^{-\frac{n+\alpha}{n-\alpha}}P_{\alpha/2}^g(\rho \phi), ~~\forall~ \rho, \phi\in C^\infty(M) ~\text{with}~\rho>0.
    \end{aligned}
\end{equation*}
In the case that $M = \mathbb{R}^n$ with the Euclidean metric $|dx|^2$, Chang and Gonz\'alez \autocite{fli_cs} illustrated that the conformal fractional Laplacian $P_{\alpha/2}^{|dx|^2_{\mathbb{R}^n}}$ coincides with the pure power fractional Laplacian $(-\Delta_{\mathbb{R}^n})^{\alpha/2}$ for $0<\alpha<2$.

Let $(\mathbb{S}^n,g_{\mathbb{S}^n})$ with $n\geq 2$ be the unit sphere equipped with the standard round metric $g_{\mathbb{S}^n}$, the conformal fractional Laplacian $P_{\alpha/2}^{g_{\mathbb{S}^n}}$
is the pull back of the fractional Laplacian $(-\Delta_{\mathbb{R}^n})^{\alpha/2}$ via the stereographic projection through
\begin{equation*}
    \begin{aligned}
        (P_{\alpha/2}^{g_{\mathbb{S}^n}}(\phi))\circ\Psi=(\det\diff \Psi)^{-\frac{n+\alpha}{2n}}(-\Delta_{\mathbb{R}^n})^{\alpha/2}((\det \diff \Psi)^{\frac{n-\alpha}{2n}}\phi\circ\Psi),~~\text{for}~\phi\in C^2(\mathbb{S}^n).
    \end{aligned}
\end{equation*}
where $\Psi$ is the inverse of the stereographic projection and $\Psi^*(g_{\mathbb{S}^n})=(\det \diff \Psi)^{\frac{2}{n}}|\diff x|^2$. 

Define $R_{\alpha/2}^{g} = P_{\alpha/2}^{g}(1)$, which is called the fractional order curvature on $\mathbb{S}^n$. Considering the following fractional Nirenberg problem using a geometric flow. Given a positive smooth function $f$ on $\mathbb{S}^n$ and $g_0\in [g_{\mathbb{S}^n}]$, we define the flow $g=g(t)$:
\begin{equation*}
    \begin{aligned}
        \frac{\partial g}{\partial t}=(\gamma f-R_{\alpha/2}^g)g,~g|_{t=0}=g_0,
    \end{aligned}
\end{equation*}
where $\gamma(t)=\frac{\int _{\mathbb{S}^n}R_{\alpha/2}^g\diff V_g}{\int_{\mathbb{S}^n}f\diff V_g}$ and $0<\alpha<2$. The parabolic equation of this flow can be rewritten as the fractional filtration equation on $\mathbb{R}^n$ by rescalings through the stereographic projection. This approach has been well used in the investigations of fractional Nirenberg problem and fractional Yamabe problem on the unit sphere $\mathbb{S}^{n}$ in \autocites{afc_cx}{fys_hp}{afy_jt}. 

In a similar way, we will consider the fractional geometric flow of the Cauchy Riemannian (CR) structure instead of the Riemannian structure. The investigations of CR geometry and sub-Riemannian geometry originated from the failure of the Riemann mapping theorem in complex analysis with several variables. The flat space in the CR geometry or sub-Riemannian geometry is the Heisenberg group $\mathbb{H}^n$. 

Let $-\mathscr{L}$ be the sub-Laplacian on the Heisenberg group. The fractional power of the operator $\mathscr{L}$ can be defined through semigroup approach, denoted by $\mathscr{L}^{\alpha/2}$. An extension technique involving $\mathscr{L}^{\alpha/2}$ akin to that of Caffarelli-Silvestre has been developed in a general setting by Stinga and Torrea in \autocite{epa_sp}. This naturally leads to a theory of existence, uniqueness and quantitative properties of the fractional filtration equations on $\mathbb{H}^n$ via the Caffarelli-Silvestre type extension in a similar way. 

However, the pure power fractional operator $\mathscr{L}^{\alpha/2}$ is not conformally covariant with respect to conformal class of pseudo-Hermitian structures, which seems to be very counterintuitive since $(-\Delta_{\mathbb{R}^n})^{\alpha/2}$ are conformally covariant. There is another pseudodifferential operator $\mathscr{L}_{\alpha/2}$ satisfying a conformally covariant formula. This is where a new story begins.

The conformally covariant operator $\mathscr{L}_{\alpha/2}$ was introduced by Branson, Fontana and Morpurgo in \autocite{mta_bt} via the spectral formula,
\begin{equation*}
    \begin{aligned}
        \mathscr{L}_{\alpha/2}=2^{\frac{\alpha}{2}}|\partial _s|^{\frac{\alpha}{2}}\frac{\Gamma(-\frac{1}{2}\mathscr{L}|\partial _s|^{-1}+\frac{1+\frac{\alpha}{2}}{2})}{\Gamma(-\frac{1}{2}\mathscr{L}|\partial _s|^{-1}+\frac{1-\frac{\alpha}{2}}{2})},~~0<\alpha<2.
    \end{aligned}
\end{equation*}
In \autocite{aep_fr}, Frank, González, Monticelli and Tan formulated the following extension problem,
\begin{equation*}
    \left\{
        \begin{array}{ll}
            \partial_{tt}U+ (1-\alpha)t^{-1}\partial_t U+t^2\partial_{ss}U+\mathscr{L}U = 0, & \text{in}~\mathbb{H}^n\times \mathbb{R}^+,  \\
            U(x,0) = f(x),     & x=(z,s) \in \mathbb{H}^n.
        \end{array}
    \right.
\end{equation*}
The authors recovered the operator $\mathscr{L}_{\alpha/2} f$ as the Neumann data of the above extension function $U$. The additional term $t^2\partial_{ss}$, which is a fourth-order term with respect to the dilations $\delta_\lambda$ on $\mathbb{H}^n$, makes the extension problem harder than the Caffarelli-Silvestre type extension. 

Let $(X^{n+1},h^+)$ be a Kähler-Einstein manifold with strictly pseudoconvex boundary $(M,[\theta])$. The conformal fractional operator $P_{\alpha/2}^{\theta}$ on $(M, \theta)$ is constructed as the Dirichlet-to-Neumann operator for a generalized eigenvalue problem (see \autocite{aep_fr}). The fractional order curvature associated to $P_{\alpha/2}^{\theta}$ is defined as $R_{\alpha/2}^{\theta}=P^{\theta}_{\alpha/2}(1)$. Here, $0 < \alpha < 2$. 

Let $(\mathbb{S}^{2n+1},\hat{\theta}_0)$ with $n\geq 1$ be the unit sphere equipped with the standard pseudo-Hermitian structure $\hat{\theta}_0$. The correspondence between the CR sphere and the Heisenberg group arises from Cayley transform, see \autocite{dga_sd}. As pointed out in \autocite{mta_bt}, $P_{\alpha/2}^{\hat{\theta}_0}$ is the pull back of the conformally covariant operator $\mathscr{L}_{\alpha/2}$ via the Cayley transformation through
\begin{equation*}
  \begin{aligned}
        (P_{\alpha/2}^{\hat{\theta}_0}(\phi))\circ\Psi_C=(\det\diff \Psi_C)^{-\frac{2(n+1)+\alpha}{4(n+1)}}\mathscr{L}_{\alpha/2}((\det \diff \Psi_C)^{\frac{2(n+1)-\alpha}{4(n+1)}}\phi\circ\Psi_C),~~\text{for}~\phi\in C^2(\mathbb{S}^{2n+1}).
  \end{aligned}
\end{equation*}
where $\Psi_C$ is the Cayley transformation and $\Psi_C^*(\hat{\theta}_0)=(|J_C|)^{\frac{1}{n+1}}\theta_0$. 

Considering the CR fractional curvature flow on $\mathbb{S}^{2n+1}$, it is defined as the evolution of the contact form $\hat{\theta}(t)$:
\begin{equation*}
    \begin{aligned}
        \frac{\partial \hat{\theta}}{\partial t }=-(Q_{\alpha/2} ^{\hat{\theta}}-\gamma f)\hat{\theta}, ~\hat{\theta}(0)=v_0^{\frac{4}{2(n+1)-\alpha}}\hat{\theta}_0.
    \end{aligned}
\end{equation*}
where $\gamma(t)=\frac{\int _{\mathbb{S}^{2n+1}}{Q_{\alpha/2} ^{\hat{\theta}}\diff V_{\hat{\theta}}}}{\int _{\mathbb{S}^{2n+1}}f\diff V_{\hat{\theta}}}$ and $0 < \alpha < 2$.
 
The Cayley transform is a natural way to transform the evolution equation of this geometric flow on the CR sphere $\mathbb{S}^{2n+1}$ to the nonlocal equation of filtration type involving $\mathscr{L}_{\alpha/2}$ on the Heisenberg group $\mathbb{H}^n$. However, the extension map corresponding to $\mathscr{L}_{\alpha/2}$ is not good enough as the 
Caffarelli-Silvestre type extension, we can not obtain analogous result of the nonlocal filtration type equation on $\mathbb{H}^n$ through the extension technique as \autocite{agf_dpa}.

On the other hand, Roncal and Thangavelu \autocite{hif_rl} established the essential pointwise integral representation of the operator $\mathscr{L}_{\alpha/2}$ using tools from noncommutative harmonic analysis. To be precise, for Schwartz function $f$,
\begin{equation*}
    \begin{aligned}
        \mathscr{L}_{\alpha/2}f(x)=C_{n,\alpha}\text{P.V.}\int _{\mathbb{H}^n}\frac{f(x)-f(y)}{|y^{-1}\cdot x|_{\mathbb{H}^n}^{Q+\alpha}}\diff \mu({y}), ~~0<\alpha<2. 
    \end{aligned}
\end{equation*}

Furthermore, as pointed out in \autocite{hif_ff}, the operator $\mathscr{L}^{\alpha/2}$ also has a pointwise integro-representation. To be precise, for the Schwartz function $f$, 
\begin{equation*}
    \begin{aligned}
        \mathscr{L}^{\alpha/2} f(x)=\text{P.V.}\int _{\mathbb{H}^n}{(f(x)-f(y))}\Tilde{R}_{-\alpha}(y^{-1}\cdot x)\diff \mu(y),~~0<\alpha<2.
    \end{aligned}
\end{equation*} 
where $\Tilde{R}_{-\alpha}(x)$ is a positive smooth function on  $\mathbb{H}^n\backslash \{0\}$ and comparable with $|x|_{\mathbb{H}^n}^{-\alpha-Q}$.

Therefore, the integro-representations of $\mathscr{L}^{\alpha/2}$ and $\mathscr{L}_{\alpha/2}$ motivate the investigations of filtration type equations involving the integro-kernels with rough estimates on $\mathbb{H}^n$. 

Let us recall the past investigations of the integro-differential equations on $\mathbb{R}^n$. Integro-differential equations naturally arise from models in physics, engineering, and finance that involve long range interactions, see for instance \autocite{fmw_tp}. They are also a natural generalization of fractional differential equations, since the fractional Laplacian $(-\Delta_{\mathbb{R}^n})^{\alpha/2}$ is a classic example of nonlocal operators with a specific integro-kernel,
\begin{equation*}
    \begin{aligned}
        (-\Delta_{\mathbb{R}^n})^{\alpha/2}f(x)=C_{n,\alpha}\text{P.V.}\int _{\mathbb{R}^n}\frac{f(x)-f(z)}{|x-z|^{n+\alpha}}\diff z
    \end{aligned}
\end{equation*}
where $C_{n,\alpha}=\frac{2^{\alpha-1}\alpha\Gamma((n+\alpha)/2)}{\pi^{n/2}\Gamma(1-\alpha/2)}$.

The existence, uniqueness and quantitative properties of the solutions to the following nonlocal filtration type equations have been investigated in \autocites{ttp_ai}{rtfp_cl}{nfe_dpa}{rtf_dpa},
\begin{equation*}
    \left\{
            \begin{array}{ll}
                \partial_t u+ \mathcal{L}u^m = 0, & x\in \mathbb{R}^n,  t>0  \\
                u(x,0) = u_0(x), & x \in \mathbb{R}^n. 
            \end{array}
    \right.
    \end{equation*}
with suitable initial data $u_0$. There is no sign restriction of $u$. The nonlocal operator $\mathcal{L}$ is defined formally as
\begin{equation*}
    \begin{aligned}
        \mathcal{L}f(x)=\text{P.V.}\int _{\mathbb{R}^n}(f(x)-f(y))J(x,y)\diff y
    \end{aligned}
\end{equation*}
with a measurable kernel $J$ satisfying  
\begin{equation*}
    \left\{
            \begin{array}{ll}
                J(x,y)=J(y,x)\geq0~~&x,y\in \mathbb{R}^n , x\neq y,\\
                \frac{1}{\Lambda|x-y|^{n+\alpha}}\leq J(x,y)\leq \frac{\Lambda}{|x-y|^{Q+\alpha}},~~&0<\alpha<2, ~~\Lambda>0.
            \end{array}
    \right.
\end{equation*} 
The integro-operator satisfying the required conditions has been greatly studied especially in the probability theory, for instance the Markov jump process and the martingale problem. For the corresponding filtration equations with these integro-operators, the critical value $m^*:=\frac{n-\alpha}{n}$ gives a natural way to develop non-identical theories for $m>m^*$ and $m<m^*$, see more details in \autocite{agf_dpa}.

In this paper, we establish the existence, uniqueness and quantitative properties of solutions to the nonlocal filtration type equations with integro-operators on the Heisenberg group $\mathbb{H}^n$.
\begin{equation}
    \left\{
            \begin{array}{ll}
                \partial_t u+ \mathcal{L}(|u|^{m-1}u) = 0, & x\in \mathbb{H}^n,  t>0  \\
                u(x,0) = u_0(x), & x \in \mathbb{H}^n. 
            \end{array}
    \right.
\end{equation}
with suitable initial data $u_0$. There is no sign restriction of $u$. Similarly, the nonlocal operator $\mathcal{L}$ is defined formally as
\begin{equation}
    \begin{aligned}
        \mathcal{L}f(x)=\text{P.V.}\int _{\mathbb{H}^n}(f(x)-f(y))J(x,y)\diff \mu(y) 
    \end{aligned}
\end{equation}
with a measurable kernel $J$ satisfying  
\begin{equation}
    \left\{
            \begin{array}{ll}
                J(x,x\cdot y)=J(x,x\cdot y^{-1})\geq0~~&x,y\in \mathbb{H}^n , x\neq y,\\
                \frac{1}{\Lambda|x^{-1}\cdot y|^{Q+\alpha}}\leq J(x,y)\leq \frac{\Lambda}{|x^{-1}\cdot y|^{Q+\alpha}},~~&0<\alpha<2, ~~\Lambda>0.
            \end{array}
    \right.
\end{equation} 
Here $|\cdot|$ is the standard homogeneous quasi-norm on $\mathbb{H}^n$ and $Q=2n+2$ is the corresponding homogeneous degree. $\diff \mu(y)$ is the volume form with respect to standard Haar measure on $\mathbb{H}^n$. Under the symmetric assumption $J(x,x\cdot y)=J(x,x\cdot y^{-1})$, the operator has a pointwise expression,
\begin{equation}
    \begin{aligned}
        \mathcal{L}f(x)=\frac{1}{2}\int_{\mathbb{H}^n}({2f(x)-f(x\cdot y)-f(x\cdot y^{-1})})J(x^{-1}\cdot y)\diff \mu(y).
    \end{aligned}
\end{equation}
for regular enough $f$. Without misunderstandings, $J(x, y)$ is denoted by $J(x^{-1}\cdot y)$. Comparably, the critical value $m^*:=\frac{Q-\alpha}{Q}$ gives a natural way to develop non-identical theories for $m>m^*$ and $m<m^*$. 

In a forthcoming paper we shall apply the results in this paper to prove the long time existence of the fractional order curvature flow on the CR sphere through Cayley transform.

\subsection{Main Results}

\begin{theorem}
    Let $m>m^*$, $0<\alpha<2$. $J$ satisfies  assumption (1.3). Then for every $u_0\in L^1(\mathbb{H}^n)$, there exists a unique weak solution of Cauchy problem (1.1). Moreover,
    \begin{enumerate}[label=(\roman*)]
        \item $\partial _t u\in L^{\infty}((\tau,\infty):L^1({\mathbb{H}^n}))$ for every $\tau>0$.
        \item Conservation of mass: $\int_{\mathbb{H}^n}u(x,t)\diff \mu=\int_{\mathbb{H}^n}u_0(x,t)\diff \mu$. 
        \item $L^p$ norm of $u$ is non-increasing in time, for each $1\leq p\leq \infty$.
        \item $L^1$ contraction property holds:
        $$\int _{\mathbb{H}^n}(u(\cdot,t)-v(\cdot,t))_+\leq \int _{\mathbb{H}^n}(u(\cdot,0)-v(\cdot,0))_+,~~\text{for} ~t\geq0.$$
        \item $L^p-L^\infty$ smoothing effects hold:
        $$\sup\limits_{x\in \mathbb{H}^n}|u(x,t)|\leq C(m,p,\alpha,Q)t^{-\gamma_p}\norm{f}_p^{\delta_p}.$$
        with $\gamma_p=(m-1+\alpha p/Q)^{-1}$ and $\delta_p=\alpha p \gamma_p/Q$ for each $1\leq p\leq \infty$.
    \end{enumerate}
\end{theorem}

\begin{remark}
    {Weak solutions will be defined in Section 2. Furthermore, if solution $u$ satisfies condition (i), it can be considered as a strong solution in some sense.}
\end{remark}

If $m \leq m^*$, we establish the existence of weak solutions corresponding to more restrictive initial data compared with $m>m^*$.

\begin{theorem}
    Let $0<m\leq m^*$, $0<\alpha<2$. $J$ satisfies  assumption (1.3). Then for every $u_0\in L^1(\mathbb{H}^n)\cap L^{p}(\mathbb{H}^n)$ with $p>p^*(m)=(1-m)Q/\alpha$, there exists a unique weak solution of the Cauchy problem (1.1). Furthermore,

    \begin{enumerate}[label=(\roman*)]
        \item $\partial _t u\in L^{\infty}((\tau,\infty):L^1({\mathbb{H}^n}))$ for every $\tau>0$.
        \item Conservation of mass: $\int_{\mathbb{H}^n}u(x,t)\diff \mu=\int_{\mathbb{H}^n}u_0(x,t)\diff \mu$ holds for $m=m^*$. Otherwise, there exists a finite time $T>0$ such that $u(x,T)\equiv 0$ for $0<m<m^*$. 
        \item $L^1$ contraction property holds; $L^p$ norm of $u$ is non-increasing in time, $1\leq p\leq \infty$.
        \item $L^p-L^\infty$ smoothing effects hold:
        $$\sup\limits_{x\in \mathbb{H}^n}|u(x,t)|\leq C(m,p,\alpha,Q)t^{-\gamma_p}\norm{f}_p^{\delta_p}.$$
        with $\gamma_p=(m-1+\alpha p/Q)^{-1}$ and $\delta_p=\alpha p \gamma_p/Q$. Here $p>p^*(m)$.
    \end{enumerate}
\end{theorem}

\begin{remark}
    {Conservation of mass for $m> \frac{n-\alpha}{n}$ on $\mathbb{R}^n$ has been formulated in \autocite{nfe_dpa}. This property also holds for $m= \frac{n-\alpha}{n}$ on $\mathbb{R}^n$ by a same proof as the case of Heisenberg group $\mathbb{H}^n$.}
\end{remark}

Furthermore, we obtain some important regularity results of solutions to the equations (1.1). The first result is continuity at the points away from $t=0$. 

\begin{theorem}
    Let $m > 0$, $0 < \alpha < 2$. J satisfies assumption (1.3). Let $u$ be the weak solution of the Cauchy problem (1.1). Then $u \in C((0,+\infty)\times\mathbb{H}^n)$. 
\end{theorem}

We also prove the Hölder regularity holds for the positive solution with uniform positive lower bound. There will be no degenerary ($m > 1$) or singularity ($m < 1$) since $u \geq \delta >0$. 

\begin{theorem}
    Let $m > 0$, $0 < \alpha < 2$. J satisfies assumption (1.3). Let $u$ be the weak solution of the Cauchy problem (1.1). If $u$ is nonnegative with uniform positive lower bound, for each $\epsilon > 0$, there exists $0 < \alpha < 1$ such that $u \in C^{\alpha}((\epsilon, +\infty) \times \mathbb{H}^n)$.
\end{theorem}

Moreover, if $m \geq 1$, the Hölder regularity will also hold for the sign-changing solutions since we can overcome the difficulty arising from the degeneracy at the vanishing points in a quantitative way. Here, the vanishing points $(x_0,t_0)$ are the points satisfying $u(x_0,t_0) = 0$. 

\begin{theorem}
    Let $m \geq 1$, $0 < \alpha < 2$. J satisfies assumption (1.3). Let $u$ be the bounded weak solution of the Cauchy problem (1.1). Then for each $\epsilon > 0$, there exists $0 < \alpha < 1$ such that $u \in C^{\alpha}((\epsilon, +\infty) \times \mathbb{H}^n)$.
\end{theorem}

\begin{remark}
    The modulus of continuity follows the method introduced by De Giorgi in his classical proof to elliptic equations, see \autocite{sdl_dgm}. The linear nonlocal operator has been considered by Caffarelli, Chan and Vasseur in \autocite{rtfp_cl} using the De Giorgi's method. Furthermore, the Hölder modulus of continuity for positive solution with uniform positive lower bound to the analogous nonlocal filtration type equation on $\mathbb{R}^n$ has been well studied in \autocites{ttp_ai}{nfe_dpa}{rtf_dpa}, by using the non-quadratic energies and the De Giorgi's method.
\end{remark}

\begin{remark}
    The new result is, for the porous medium type equations ($m\geq 1$), the uniform H\"older continuity holds for all the points including the vanishing points. We constructed the iterated sequence of solutions with iterated nonlinearity functions in decreasing space-time cylinders. As the function value $u(\cdot,t)$ approaches zero, the derivatives of the iterated nonlinearity functions are not uniformly bounded, which results in the degeneracy of the oscillation reduction. To overcome this difficulty, we work on cylinders suitably scaled to reflect in a precise quantitative way the degeneracy at the vanishing points using the idea from \autocites{hef_de}{tpm_vj}. 
\end{remark}

\section{Preliminaries}
\subsection{Heisenberg group $\mathbb{H}^n$}
We recall some definitions and a few well-known facts concerning with the Heisenberg group $\mathbb{H}^n$. For further details we refer to the book by Fischer and Ruzhansky, \autocite{qon_fv}.

We identify the Heisenberg group $\mathbb{H}^n$ with $\mathbb{R}^{2n+1}$. An element in the Heisenberg group $\mathbb{H}^n$ is denoted by 
$$x:=(z,s)=(\xi_1,\cdots,\xi_n,\eta_1,\cdots,\eta_n,s).$$
For any $x,x'\in \mathbb{R}^{2n+1}$, the group multiplication is given by
\begin{equation}
    \begin{aligned}
        x\cdot x':&=(\xi+\xi',\eta+\eta',s+s'+2\langle\eta,\xi'\rangle-2\langle\xi,\eta'\rangle)\nonumber \\
        &=(\xi_1+\xi'_1,\cdots,\xi_n+\xi'_n,\eta_1+\eta'_1,\cdots,\eta_n+\eta'_n,s+s'+2\sum_{i=1}^n(\eta_i\xi'_i-\xi_i\eta'_i)).
    \end{aligned}
\end{equation}
where $\langle \cdot,\cdot,\rangle$ is the standard inner product in $\mathbb{R}^n$.

The neutral element of $\mathbb{H}^n$ is $0_{\mathbb{H}^n}=(0_{\mathbb{R}^n},0_{\mathbb{R}^n},0)$ and the inverse $(z,s)^{-1}$ of the element $(z,s)$ is $(-z,-s)$. Define the dilations $\delta_\lambda$ on $\mathbb{H}^n$,
\begin{equation}
    \begin{aligned}
        \delta_\lambda(x)=(\lambda\xi,\lambda\eta,\lambda^2 s),~~\lambda>0.\nonumber
    \end{aligned}
\end{equation}

The Heisenberg group $\mathbb{H}^n$ can be considered as a sub-Riemannian manifold. The orthonormal basis of the Heisenberg group $\mathbb{H}^n$ is given by
\begin{equation}
    \begin{aligned}
        X_j:=\frac{\partial}{\partial \xi_j}+2\eta_j\frac{\partial }{\partial s}, ~~Y_j:=\frac{\partial}{\partial \eta_j}-2\xi_j\frac{\partial }{\partial s},~~1\leq j\leq n,~~T=\frac{\partial}{\partial s}.\nonumber
    \end{aligned}
\end{equation}
which is the Jacobian basis of the Heisenberg Lie algebra $\mathbf{h}^n$ of $\mathbb{H}^n$. Here,
\begin{equation}
    \begin{aligned}
        \text{rank}(\text{Lie}\{X_1,\cdots,X_n,Y_1,\cdots,Y_n,T\}(0_{\mathbb{H}^n}))=2n+1.\nonumber
    \end{aligned}
\end{equation}
This shows that $\mathbb{H}^n$ is a Carnot group with the following stratification
\begin{equation}
    \begin{aligned}
        \mathbf{h}^n=\text{span}\{X_1,\cdots,X_n,Y_1,\cdots,Y_n\}\oplus\text{span}\{T\}.\nonumber
    \end{aligned}
\end{equation}

Let $(|z|^4+t^2)^{1/4}=((|\xi|^2+|\eta|^2)^2+t^2)^{1/4}$ be the homogeneous quasi-norm on $\mathbb{H}^n$, which is denote by $|{\cdot}|_{\mathbb{H}^n}$. The Kor\'anyi-Cygan metric $d_{\mathbb{H}^n}:\mathbb{H}^n\times \mathbb{H}^n\rightarrow\mathbb{R}$ is defined as
\begin{equation}
    \begin{aligned}
        d_{\mathbb{H}^n}(x,x')=|x^{-1}\cdot x'|_{\mathbb{H}^n}.\nonumber
    \end{aligned}
\end{equation}
For the sake of readability, the subscript will be often omitted without causing misunderstandings. The Lebesgue measure $\diff \mu(\cdot)$ on $\mathbb{R}^{2n+1}$ will be the Haar measure on $\mathbb{H}^n$ which is uniquely defined up to some positive constant. Let $Q=2n+2$ be the homogeneous degree corresponding to the automorphisms $(\delta_{\lambda})_{\lambda>0}$. 

For any fixed $x_0\in \mathbb{H}^n$ and $R>0$, denote with $B_R(x_0)$ the Kor\'anyi ball with center $x_0$ and radius $R$ defined as
\begin{equation}
    \begin{aligned}
        B_R(x_0):=\{x\in \mathbb{H}^n:|x_0^{-1}\cdot x|\leq R\}.\nonumber
    \end{aligned}
\end{equation}

Let $\Omega\subset \mathbb{H}^n$ be a domain, for each $f\in C^2(\Omega;\mathbb{R})$, the negative sublaplacian $\mathscr{L}$ is defined by,
\begin{equation}
    \begin{aligned}
        \mathscr{L} f:=-\sum _{i=1}^{n}X_i^2 f-\sum _{i=1}^{n}Y_i^2 f.\nonumber
    \end{aligned}
\end{equation}

The pure fractional powers of $\mathscr{L}^{\alpha/2}$ has been well defined and studied in \autocites{sea_fg,hso_fg,hif_ff}. For $\alpha>0$, the operator $\mathscr{L}^{\alpha/2}$ can be written as
\begin{equation}
    \begin{aligned}
        \mathscr{L}^{\alpha/2}=\int_0^{+\infty}\lambda^{\alpha/2}\diff E(\lambda)\nonumber
    \end{aligned}
\end{equation}
with domain 
$$W^{\alpha/2,2}(\mathbb{H}^n):=\{u\in L^2(\mathbb{H}^n): \norm {\mathscr{L}^{\alpha/4}u}_{L^2(\mathbb{H}^n)}< \infty\}.$$
Here $\diff E(\lambda)$ is the spectral resolution of $\mathscr{L}$ in $L^2(\mathbb{H}^n)$.

The fractional homogeneous Sobolev space $\dot{H}^{\alpha/{2}}(\mathbb{H}^n)$ is defined as the completion of $C_0^\infty (\mathbb{H}^n)$ with the norm 
$$\norm{u}_{\dot{H}^{\alpha/{2}}(\mathbb{H}^n)} = \norm {\mathscr{L}^{\alpha/{4}}u}_{L^2(\mathbb{H}^n)}.$$  

\subsection{Weak Solutions}
For simplicity, we denote $\norm{\cdot}_{L^p(\mathbb{H}^n)}$ by $\norm{\cdot}_p$ for $1\leq p\leq \infty$ and use the simplified notation $u^m$ instead of $|u|^{m-1}u$ for sign changing solutions. In order to define weak solution, we consider the bilinear Dirichlet form,
$$ \mathcal{E}_J(u,v)=\frac{1}{2}\int _{\mathbb{H}^n}\int _{\mathbb{H}^n}(u(x)-u(y))(v(x)-v(y))J(x,y)\diff \mu(x)\diff \mu(y).$$
For kernels satisfying the symmetry condition $J(x,x\cdot y)=J(x,x\cdot y^{-1})$ and functions $f,g\in C_0^2(\mathbb{H}^n)$, we have
$$<\mathcal{L}f,g>=\mathcal{E}_J(u,v).$$

Denote the space $\dot{\mathcal{H}}_{\mathcal{L}}(\mathbb{H}^n)$ as the closure of space $C_0^{\infty}(\mathbb{H}^n)$ under the seminorm $\mathcal{E}_J(u,u)^{\frac{1}{2}}$. We also define $\mathcal{H}_{\mathcal{L}}(\mathbb{H}^n)=\{f\in L^2(\mathbb{H}^n):\mathcal{E}_J(f,f)<\infty\}$. The condition (1.3) implies
\begin{equation}
    \begin{aligned}
        c_1(n,\alpha,\Lambda)\mathcal{E}_J(u,u)\leq \norm {\mathscr{L}^{\alpha/{4}}u}_{2}^2\leq c_2(n,\alpha,\Lambda)\mathcal{E}_J(u,u).
    \end{aligned}
\end{equation}

Now we are ready to define the weak solutions of Cauchy problem (1.1).
\begin{definition}
A function $u$ is a weak solution to equation (1.1) if:
\begin{itemize}
    \item $u\in C([0,\infty):L^1(\mathbb{H}^n))$ \rm{and} $u^m\in L^2_{\text{loc}}((0,\infty):\dot{\mathcal{H}}_{\mathcal{L}}(\mathbb{H}^n))$;
    \item $\int _0^\infty \int _{\mathbb{H}^n} u \partial _t \zeta-\int_0^\infty \mathcal{E}_J(u^m,\zeta)=0 $ \rm{for each} $\zeta\in C^2_{c}(\mathbb{H}^n\times(0,\infty))$;
    \item $u(x,0)=u_0(x)$ \rm{almost everywhere}.
\end{itemize}
\end{definition}

\noindent If the equation (1.1) is rewritten equivalently as 
\begin{equation}
    \left\{
        \begin{array}{ll}
            \partial_t w^{1/m}+ \mathcal{L}w = 0, & x\in \mathbb{H}^n,  t>0  \\
            w(x,0) = w_0(x) = u_0^m(x), & x \in \mathbb{H}^n. 
        \end{array}
    \right.
\end{equation}
with $w_0^{1/m}\in L^1(\mathbb{H}^n)$ and $w=u^m$. Equivalently, we define the weak solutions to the Cauchy problem (2.2).
  \begin{definition}
A function $w$ is a weak solution to equation (2.2) if:
\begin{itemize}
    \item $w^{1/m}\in C([0,\infty):L^1(\mathbb{H}^n))$ \rm{and} $w\in L^2_{\text{loc}}((0,\infty):\dot{\mathcal{H}}_{\mathcal{L}}(\mathbb{H}^n))$;
    \item $\int _0^\infty \int _{\mathbb{H}^n} w^{1/m} \partial _t \zeta-\int_0^\infty \mathcal{E}_J(w,\zeta)=0 $ for each $\zeta\in C^2_{c}(\mathbb{H}^n\times(0,\infty))$;
    \item $w(x,0)=w_0(x)$ almost everywhere.
\end{itemize}
\end{definition}

\subsection{Some Inequalities}
We quote or prove some important inequalities which are useful in the following sections. The Stroock Varopoulos inequality will be proved by the extension map which can be used to reconstruct the operator $\mathscr{L}^{\alpha/2}(u)$, see \autocites{hif_ff,epa_sp}. 
\begin{lemma}
    Let $0<\alpha<2$. If $\psi_1(v), v \in \dot{\mathcal{H}}^{\alpha/2}(\mathbb{H}^n) $ and $(\Psi')^2\leq \psi_1'(v)$, then
    \begin{equation*}
        \begin{aligned}
            \int_{\mathbb{H}^n}(\mathscr{L}^{\alpha/2}\psi_1(v))(v)\geq \int_{\mathbb{H}^n}(\mathscr{L}^{\alpha/2}\Psi(v))\Psi(v).
        \end{aligned}
    \end{equation*}
\end{lemma}

\begin{proof}
    Recall that there exists a Caffarelli-Silvestre type extension for the pure power fractional operator $\mathscr{L}^{\alpha/2}$, see \autocites{hif_ff,epa_sp}. Define the space $\dot{X}^{\alpha/2}(\mathbb{H}^n\times \mathbb{R}^{+})$ as the completion of $C^{\infty}(\mathbb{H}^n\times[0,\infty))$ under the seminorm
        $$\norm{w}_{\dot{X}^{\alpha/2}(\mathbb{H}^n\times \mathbb{R}^{+})}=\int_{0}^{\infty}\int_{\mathbb{H}^n}y^{1-\alpha}|\nabla_H w(x)|^2\diff\mu(x)\diff y.$$
    Here $\nabla_H$ is the horizontal gradient on $\mathbb{H}^n\times(0,\infty)$ defined by,
    \begin{equation}
        \begin{aligned}
            \nabla _{H}f:=\sum_{i=1}^{n}(X_i f)X_i+\sum_{i=1}^{n}(Y_i f)Y_i+(\partial_y f) \partial_y .\nonumber
        \end{aligned}
    \end{equation}
    and
    \begin{equation}
        \begin{aligned}
            |\nabla _{H}f|^2:=\sum_{i=1}^{n}|X_i f|^2+\sum_{i=1}^{n}|Y_i f|^2+|\partial_y f|^2.\nonumber
        \end{aligned}
    \end{equation}
    For every $u\in \dot{H}^{\alpha/2}(\mathbb{H}^n)$, there exists a norm-preserving extension map $E:\dot{H}^{\alpha/2}(\mathbb{H}^n)\rightarrow \dot{X}^{\alpha/2}(\mathbb{H}^n\times \mathbb{R}^{+})$ satisfying,
    \begin{equation*}
        \left\{
            \begin{array}{ll}
                (y^{1-\alpha}\mathscr{L}+y^{1-\alpha}\partial_{yy}+(1-\alpha)y^{-\alpha}\partial_y)E(u)(x,y)=0,&(x,y)\in \mathbb{H}^n\times \mathbb{R}^+ \\
                E(u)(x,0)=u(x), &x\in \mathbb{H}^n.
            \end{array}
        \right.
    \end{equation*}
    Moreover, the pure power opertaor $\mathscr{L}^{\alpha/2}$ can be reconstructed as the Neumann data of the extension map,
    \begin{equation*}
        \mathscr{L}^{\alpha/2}u(x)=-\mu_{\alpha}\lim\limits_{y\rightarrow 0^+}y^{1-\alpha}\frac{\partial {E(u)}}{\partial y}, ~~\mu_{\alpha}=2^{\alpha-1}\Gamma({\alpha/2})\Gamma(1-\alpha/2).
    \end{equation*}
    Then,
    \begin{equation*}
        \begin{aligned}
            \int_{\mathbb{H}^n}(\mathscr{L}^{\alpha/2}\psi_1(v))(v)&=\mu_{\alpha}\int_0^{\infty}\int_{\mathbb{H}^n}y^{1-\alpha}\langle\nabla_H E(\psi_1(v)),\nabla_H E(v) \rangle \diff \mu(x)\diff y\\
            &=\mu_{\alpha}\int_0^{\infty}\int_{\mathbb{H}^n}y^{1-\alpha}\langle\nabla_H \psi_1(E(v)),\nabla_H E(v) \rangle \diff \mu(x)\diff y\\
            &\geq\mu_{\alpha}\int_0^{\infty}\int_{\mathbb{H}^n}y^{1-\alpha}\langle\nabla_H E(\Psi(v)),\nabla_H E(\Psi(v)) \rangle \diff \mu(x)\diff y\\
            &=\int_{\mathbb{H}^n}(\mathscr{L}^{\alpha/2}\Psi(v))\Psi(v).
        \end{aligned}
    \end{equation*}
\end{proof}
% Here 可以 LPsi_1 Psi_2, 只需要 Psi_2(v)=w. then Psi_1(v)=Psi_1 Psi_2^{-1}W.

From the equivalence of the norm:  $c_1(n,\alpha,\Lambda)\mathcal{E}_J(u,u)\leq \norm {\mathscr{L}^{\alpha/{4}}u}_{2}^2\leq c_2(n,\alpha,\Lambda)\mathcal{E}_J(u,u)$, we obtain
\begin{lemma}
    Let $0<\alpha<2$, then
    \begin{equation}
        \begin{aligned}
            \mathcal{E}_J(\psi_1(v),v)\geq c(n,\alpha,\Lambda)\mathcal{E}_J(\Psi(v),\Psi(v)),
        \end{aligned}
    \end{equation}
    whenever $\psi_1(v), v \in \dot{\mathcal{H}}_{\mathcal{L}}(\mathbb{H}^n)= \dot{\mathcal{H}}^{\alpha/2}(\mathbb{H}^n)$ and $(\Psi')^2\leq \psi_1'(v).$
\end{lemma}

The Hardy-Littlewood-Sobolev's inequality also holds for the pure power fractional operator $\mathscr{L}^{\alpha/2}$ on $\mathbb{H}^n$, which had been proved in \cite{sea_fg},
\begin{lemma}
    For each $v$ such that $\mathscr{L}^{\alpha/2}v\in L^r(\mathbb{H}^n)$, $1<r<Q/\alpha$, $0<\alpha<2$, the following inequality holds:
    \begin{equation}
        \begin{aligned}
            \norm{v}_{{r_1}}\leq C(Q,r,\alpha)\norm{\mathscr{L}^{\alpha/2}v}_{r}.
        \end{aligned}
    \end{equation}
    Here $r_1=\frac{Qr}{Q-\alpha r}$.
\end{lemma}

By the Hardy-Littlewood-Sobolev's inequality and Hölder's inequality, we obtain, 
\begin{lemma}
    Let $p>1$, $r>1$, $0<\alpha<2$. For each $v\in L^p(\mathbb{H}^n)$ with $\mathscr{L}^{\alpha/2}v\in L^r(\mathbb{H}^n)$, we have
    \begin{equation}
        \begin{aligned}
            \norm{v}_{{r_2}}^{\gamma+1}\leq C(p,r,\alpha,Q)\norm{\mathscr{L}^{\alpha/2}v}_{r}\norm{v}_{p}^{\gamma}.
        \end{aligned}
    \end{equation}
    where $r_2=\frac{Q(rp+r-p)}{r(Q-\alpha)}$, $\gamma=\frac{p(r-1)}{r}$.
\end{lemma}

% \noindent Furthermore, the following Poincar\'e inequality holds in bounded domain $\Omega$,
% \begin{lemma}
%   For every $v\in {\mathcal{H}}_{\mathcal{L},0}(\Omega)$, then
%   \begin{equation}
%       \begin{aligned}
%       \mathcal{E}_J(v,v)\geq C(\alpha,Q,\Omega,\Lambda)\norm{f}_2^2.
%       \end{aligned}
%   \end{equation}
% \end{lemma}
%bounded doamin weak solution 对应哪一个呢，还有就是有什么用么？

\section{Existence and Uniqueness of Weak Solutions}

\subsection{Existence of Weak Solutions for Bounded Initial Data}
Crandall and Liggett \autocite{gos_cm} introduced a useful method to construct mild solutions, which is called Implicit Time Discretization. In order to apply the Crandall-Liggett's Theorem, we prove the existence of weak solution to the following equation,
\begin{equation}
    \begin{aligned}
        v^{1/m}+\mathcal{L}v=g ~~\text{in} ~~\mathbb{H}^n, ~~g\in L^1(\mathbb{H}^n)\cap L^{\infty}(\mathbb{H}^n).
    \end{aligned}
\end{equation}
and the contractivity of solutions in $L^1(\mathbb{H}^n)$. 
\begin{proposition}
    For each $g\in L^1(\mathbb{H}^n)\cap L^{\infty}(\mathbb{H}^n)$, there exists a unique weak solution $v\in \dot{\mathcal{H}}_{\mathcal{L}}(\mathbb{H}^n)$ to the equation (3.1), such that the following weak formulation holds,
    \begin{equation}
        \begin{aligned}
            \mathcal{E}_J(v,\zeta)+\int_{\mathbb{H}^n}v^{1/m}\zeta-\int_{\mathbb{H}^n}g\zeta=0, ~~\text{for}~\text{every}~\zeta\in C_0^\infty(\mathbb{H}^n).
        \end{aligned}
    \end{equation}
    Furthermore, $\norm{v^{1/m}}_\infty\leq \norm{g}_\infty$ and $\norm{v^{1/m}}_1\leq \norm{g}_1$. If $v$ and $\Tilde{v}$ are solutions corresponding to $g$ and $\Tilde{g}$ respectively, the following T-contraction inequality holds,
    \begin{equation}
        \begin{aligned}
            \int _{\mathbb{H}^n}[v^{1/m}(x)-\Tilde{v}^{1/m}(x)]_+\diff \mu(x)\leq \int _{\mathbb{H}^n}[g(x)-\Tilde{g}(x)]_+\diff \mu(x).
        \end{aligned}
    \end{equation}
\end{proposition}

\begin{proof}
    The weak solution is the minimizer of the following convex functional
    $$J(v)=\frac{1}{2}\mathcal{E}_J(v,v)+\int_{\mathbb{H}^n}\frac{m}{m+1}|v|^{\frac{1}{m}+1}-\int _{\mathbb{H}^n}v g.$$
    for each $v\in \dot{\mathcal{H}}_{\mathcal{L}}(\mathbb{H}^n)$.
    Since
    \begin{equation*}
        \begin{aligned}
            |\int _{\mathbb{H}^n}v g|&\leq \epsilon \norm{v}_{\frac{2Q}{Q-\alpha}}^2+C(\epsilon)\norm{g}_{\frac{2Q}{Q+\alpha}}^2\\
            &\leq C(Q,\alpha)\epsilon \norm{v}_{\dot{H}^{\alpha/2}_0}^2+C(\epsilon)\norm{g}_{\frac{2Q}{Q+\alpha}}^2\\
            &\leq C(Q,\alpha,\Lambda)\epsilon \mathcal{E}_J(v,v)+C({\epsilon})\norm{g}_{\frac{2Q}{Q+\alpha}}^2.
        \end{aligned}
    \end{equation*}
    Therefore, $J(v)$ is coercive, convex and lower semi-continuous and there exists a unique weak solution $v$ to equation (3.1). Let $v$ and $\Tilde{v}$ be two weak solutions of equation (3.1) corresponding to the inhomogeneous terms $g$ and $\Tilde{g}$. We use $p_n(v-\Tilde{v})$ as test functions in the weak formulation. Here $p_n$ is the smooth approximation of sign function with $0\leq p_n(\cdot)\leq 1$ and $p_n'(\cdot)\geq 0$. We obtain
    \begin{equation}
        \begin{aligned}
            \int_{\mathbb{H}^n}(v^{1/m}-\Tilde{v}^{1/m})p_n(v-\Tilde{v})+\mathcal{E}_J(v-\Tilde{v},p_n(v-\Tilde{v}))=\int_{\mathbb{H}^n}(g-\Tilde{g})p_n(v-\Tilde{v}).\nonumber
        \end{aligned}
    \end{equation}
    Passing to the limit, 
    \begin{equation}
        \begin{aligned}
            \int_{\mathbb{H}^n}(v^{1/m}-\Tilde{v}^{1/m})_+\leq \int_{\mathbb{H}^n}(g-\Tilde{g})\text{sign}(v-\Tilde{v})\leq \int_{\mathbb{H}^n}(g-\Tilde{g})_+.\nonumber
        \end{aligned}
    \end{equation}
    Here we apply the Stroock-Varopoulos inequality to obtain
    \begin{equation}
        \begin{aligned}
            \mathcal{E}_J(v-\Tilde{v},p_n(v-\Tilde{v}))\geq 0.\nonumber
        \end{aligned}
    \end{equation}
    since $p_n'(\cdot)\geq 0$. Furthermore, $v=\norm{g}_{\infty}^{m}$ is a supersolution to equation (3.1), we get
    \begin{equation}
        \begin{aligned}
            \norm{v^{1/m}}_{\infty}\leq \norm{g}_{\infty}.\nonumber
        \end{aligned}
    \end{equation}
    Choosing $v$ as the test function, we obtain
    \begin{equation}
        \begin{aligned}
            \mathcal{E}_J(v,v)\leq \int _{\mathbb{H}^n}g v \leq C.
        \end{aligned}
    \end{equation}
\end{proof}

\noindent Equivalently, define $v=u^m$, we obtain
\begin{proposition}
    For every $g\in L^1(\mathbb{H}^n)\cap L^{\infty}(\mathbb{H}^n)$, there exists a unique weak solution $u^m\in \dot{\mathcal{H}}_{\mathcal{L}}(\mathbb{H}^n)$ to the equation $u+\mathcal{L}u^m=g$ in $\mathbb{H}^n$, such that the following weak formulation holds,
    \begin{equation}
        \begin{aligned}
            \mathcal{E}_J(u^m,\zeta)+\int_{\mathbb{H}^n}u\zeta-\int_{\mathbb{H}^n}g\zeta=0, ~~\text{for}~\text{every}~\zeta\in C_0^\infty(\mathbb{H}^n).
        \end{aligned}
    \end{equation}
    Furthermore, $\norm{u}_\infty\leq \norm{g}_\infty$ and $\norm{u}_1\leq \norm{g}_1$. If $v$ and $\Tilde{v}$ are two solutions corresponding to $g$ and $\Tilde{g}$ respectively, the following T-contraction inequality holds,
    \begin{equation}
        \begin{aligned}
            \int _{\mathbb{H}^n}[u(x)-\Tilde{u}(x)]_+\diff \mu(x)\leq \int _{\mathbb{H}^n}[g(x)-\Tilde{g}(x)]_+\diff \mu(x).
        \end{aligned}
    \end{equation}
\end{proposition}

From the above Proposition, the operator $\mathcal{L}$ is accretive and satisfies the rank condition required in the Crandall-Liggett's theorem. We apply this theorem to equation (1.1) to obtain the so-called mild solution. Moreover, it is a weak solution. 

\begin{proposition}
    For each $u_0\in L^1(\mathbb{H}^n)\cap L^{\infty}(\mathbb{H}^n)$, there exists a weak solution $u\in C([0,\infty):L^1(\mathbb{H}^n))$ to equation (1.1). Moreover, $\norm{u}_\infty\leq \norm{u_0}_\infty$ and $\norm{u}_1\leq \norm{u_0}_1$. In addition, if $u$ and $\Tilde{u}$ are two such solutions corresponding to $u_0$ and $\Tilde{u_0}$, the following T-contraction inequality holds, 
    \begin{equation}
        \begin{aligned}
            \int _{\mathbb{H}^n}[v^{1/m}(x)-\Tilde{v}^{1/m}(x)]_+\diff \mu(x)\leq \int _{\mathbb{H}^n}[u_0(x)-\Tilde{u}_0(x)]_+\diff \mu(x).
        \end{aligned}
    \end{equation}
\end{proposition}

\begin{proof}
    For each $T>0$, we divide interval $[0,T]$ into $N$ subintervals. Let $\epsilon=T/N$, we construct the piecewise constant function $u_{\epsilon,k}$ in each interval $(t_{k-1},t_k]$, where $t_k=k\epsilon$ ($k=1,2,\cdots,N$) as the solutions to the following elliptic equation,  
    \begin{equation}
        \begin{aligned}
            \mathcal{L}u_{\epsilon,k}^m+\frac{1}{\epsilon}(u_{\epsilon,k}-u_{\epsilon,k-1})=0.
        \end{aligned}
    \end{equation}
    with $u_{\epsilon,0}=u_0$. By Crandall-Liggett theorem, $u_\epsilon$ converges in $L^1(\mathbb{H}^n)$ to some function $u\in C([0,\infty):L^1(\mathbb{H}^n))$. Furthermore, $\norm{u_\epsilon}_\infty\leq \norm{u_0}_\infty$ inherited from the elliptic equations. And $u_\epsilon$ converges in weak$^*$ topology to $u\in L^\infty(\mathbb{H}^n\times[0,\infty))$. Here $u$ is the mild solution to Cauchy problem (1.1). Multiplying the equation (3.8) by $u_{\epsilon,k}^m$ and applying Young's inequality, we get
    \begin{equation}
        \begin{aligned}
            \int _0^T\mathcal{E}_J(u_\epsilon^m, u_\epsilon^m)\diff t\leq \frac{1}{m+1}\int_{\mathbb{H}^n} |u_0|^{m+1}\diff \mu(x).\nonumber
        \end{aligned}
    \end{equation}
    Passing to the limit, we obtain
    \begin{equation}
        \begin{aligned}
            \int _0^T\mathcal{E}_J(u^m, u^m)\diff t\leq \frac{1}{m+1}\int_{\mathbb{H}^n} |u_0|^{m+1}\diff \mu(x).
        \end{aligned}
    \end{equation}
    Passing to the limit, we obtain the parabolic weak formulation, see the analogues in \autocite{agf_dpa}.
\end{proof}

\subsection{Uniqueness of Solutions}

In this section, we prove the uniqueness of the constructed solutions. If $m > m^*$, the weak solution defined in section 2 is always unique. However, if $m\leq m^*$, we need a stronger assumption of solutions to guarantee the uniqueness.

\begin{proposition}
    Let $u_0\in L^1(\mathbb{H}^n)$ and $m>m^*$, Problem (1.1) has at most one weak solution. 
\end{proposition}

\begin{proof}
    Following the analogues in \cite{agf_dpa}, we claim that 
    \begin{equation}
        \begin{aligned}
            u\in L^{m+1}(\mathbb{H}^n\times(0,T)) ~\text{if} ~m>m^*.\nonumber 
        \end{aligned}
    \end{equation}
    By inequalites (2,3), (2.4) and Hölder's inequality, we get
    \begin{equation}
        \begin{aligned}
            \nonumber \int _0^T\int_{\mathbb{H}^n}|u|^{m+1}\diff \mu(x)\diff t&\leq \int_0^T (\int _{\mathbb{H}^n}|u|\diff \mu(x))^{\beta}((\int_{\mathbb{H}^n}|u|^{\frac{2 Q m}{Q-\alpha}}\diff \mu(x))^{1-\beta})\diff t\\
            &\leq C(T)\max_{t\in[0,T]}\norm{u(\cdot,t)}_1^\beta [\int _0^T\norm{u^m}^2_{\frac{2Q}{Q-\alpha}}\diff t]^{1-\gamma}\\
            &\leq C(T)[\int_0^T\mathcal{E}_J(u^m,u^m)\diff t]^{1-\gamma}
        \end{aligned}
    \end{equation}
    where $\beta=\frac{Q(m-1)+\alpha(m+1)}{Q(2m-1)+\alpha}$ and $\gamma=\frac{Q(m-1)+\alpha}{Q(2m-1)+\alpha}$. 

    Define $\psi (x,t)=\int _t^T(u^m-\Tilde{u}^m)(x,s)\diff s $ for $0\leq t\leq T$ and $\psi=0$ for $t\geq T$. Let $u$ and $\Tilde{u}$ be two weak solutions to equation (1.1) with $u(\cdot,0)=\Tilde{u}(\cdot,0)=u_0(\cdot)$. From the parabolic weak formulation, we have
    \begin{equation}
        \begin{aligned}
            \int_0^T\int_{\mathbb{H}^n}&(u-\Tilde{u})(x,t)(u^m-\Tilde{u}^m)(x,t)\diff\mu(x)\diff t\\
            &+\frac{1}{2}\mathcal{E}_J(\int_0^T(u^m-\Tilde{u}^m)(x,s)\diff s,\int_0^T(u^m-\Tilde{u}^m)(x,s)\diff s)=0.\nonumber
        \end{aligned}
        %here we expand all the integrals including the integral definition of mathcal{E}_j
    \end{equation}
    Therefore, we prove $u=\Tilde{u}$ if they have the same initial data $u_0$. 
\end{proof}

Considering the case of $m\leq m^*$, we prove the uniqueness of the weak solutions which is strong in the sense;
\begin{equation}
    \begin{aligned}
        \partial _t u\in L^{\infty}((\tau,\infty):L^1(\mathbb{H}^n)),\rm{  ~~for ~every ~} \tau>0.
    \end{aligned}
\end{equation}

We prove the following comparison theorem as \autocite{afp_dpa} to obtain the uniqueness of the weak solutions satisfying (3.10), which are called strong solutions throughout this paper. 

\begin{proposition}
    Let $m>0$, if $u_1$ and $u_2$ are strong solutions to Problem (1.1) with initial data $u_{0,1}, ~u_{0,2}\in L^1(\mathbb{H}^n)$, then for $0\leq t_1\leq t_2$, we have
    \begin{equation}
        \begin{aligned}
            \int _{\mathbb{H}^n}(u_1-u_2)_+(x,t_2)\diff \mu(x)\leq \int _{\mathbb{H}^n}(u_{0,1}-u_{0,2})_+(x,t_1)\diff \mu(x)
        \end{aligned}
    \end{equation}
\end{proposition}

\begin{proof}
    Let $p_n$ be the smooth approximation of sign function with $0 \leq p_n(\cdot) \leq 1$ and $p_n'(\cdot)\geq 0$. Using $p_n(u_1 - u_2)$ as the test function in the parabolic weak formulation, for $0 < t_1 < t_2$, we have
    \begin{equation*}
        \begin{aligned}
            \int_{t_1}^{t_2}\int_{\mathbb{H}^n} \partial_t(u_1 - u_2) p_n(u_1 - u_2) \diff \mu(x) \diff t = \int_{t_1}^{t_2} \mathcal{E}_J(u_1 - u_2, p_n(u_1 - u_2))\diff t.
        \end{aligned}
    \end{equation*}
    Using Stroock-Varopoulos inequality, we have
    \begin{equation*}
        \begin{aligned}
            \int_{t_1}^{t_2} \mathcal{E}_J(u_1 - u_2, p_n(u_1 - u_2))\diff t \leq 0. 
        \end{aligned}
    \end{equation*}
    Hence, we obtain
    \begin{equation*}
        \begin{aligned}
            \int_{t_1}^{t_2}\int_{\mathbb{H}^n} \partial_t(u_1 - u_2) p_n(u_1 - u_2) \diff \mu(x) \diff t \leq 0
        \end{aligned}
    \end{equation*}
    Passing to the limit, we have
    \begin{equation*}
        \begin{aligned}
            \int_{\mathbb{H}^n}(u_1 - u_2)_+(x,t_2)\diff \mu(x) \leq \int_{\mathbb{H}^n}(u_1 - u_2)_+(x,t_1)\diff \mu(x)
        \end{aligned}
    \end{equation*}
    Let $t_1$ approach zero, we prove this inequality for $t_1=0$. 
\end{proof}

\section{Strong solutions}
In this section, we prove the constructed weak solutions are indeed strong solutions. First, we prove that $\partial_t u$ is a bounded radon measure. 

\begin{lemma}
    Assume $m\neq 1$, let $u$ be a weak solution to equation (1.1), then $\partial _t u $ is a bounded radon measure. 
\end{lemma}
\begin{proof}
    %Following proof of Lemma 8.5 in \cite{tpm_vj}, 
    If $u$ is a solution with initial data $u_0$ and $\lambda$ is a positive constant, then 
    $$\Tilde{u}(x,t)=\lambda u(x,\lambda^{m-1}t).$$
    is the weak solution with data $\Tilde{u}_0(x)=\lambda u_0(x)$. Now fix $t>0$ and $h>0$ and put $\lambda^{m-1}t=t+h$ so that $\lambda>1$. By $L^1$ contractivity estimate, we get
    \begin{equation}
        \begin{aligned}
            \norm{u(x,t+h)-u(x,t)}_1\leq 2\frac{\norm{u_0}_1}{(m-1)t}(h+o(h)).
        \end{aligned}
    \end{equation}
    Therefore, $\partial _t u$ is a finite radon measure. 
\end{proof}

First, we prove $\partial _t u^{\frac{m+1}{2}} \in L_{\text{loc}}^2(\mathbb{H}^n\times(0,\infty))$ by the method of Steklov averages. 

\begin{lemma}
    The function $z=u^{\frac{m+1}{2}}$ has the property that $\partial _t z \in L_{\text{loc}}^2(\mathbb{H}^n\times(0,\infty))$. 
\end{lemma}

\begin{proof}
    Since $\partial _t u^m$ may not be a function, it's not a reasonable test function. The Steklov averages can be used to overcome the difficulties. For any $g\in L^1(\mathbb{H}^n\times (0,+\infty))$, we define the Steklov average
    \begin{equation}
        \begin{aligned}
            g^h(x,t)=\frac{1}{h}\int_{t}^{t+h}g(x,s)\diff s.\nonumber
        \end{aligned}
    \end{equation}
    We have
    \begin{equation}
        \begin{aligned}
            \delta^h g(x,t)=\partial _t g^h(x,t):=\frac{g(x,t+h)-g(x,t)}{h} ~~\text{almost}~\text{everywhere}.\nonumber
        \end{aligned}
    \end{equation}
    Since $\partial _t u^h \in L^1(\mathbb{H}^n\times (0,+\infty))$, we rewrite the weak formulation of equation (1.1) as
    \begin{equation}
        \begin{aligned}
            \int_0^{\infty}\int_{\mathbb{H}^n}\partial _t u^h \zeta\diff \mu(x)\diff t+\int_0^{\infty}\mathcal{E}_J((u^m)^h,\zeta)\diff t=0 ~~\text{for}~ \text{each}~ \zeta\in C^2_{c}(\mathbb{H}^n\times(0,\infty)).
        \end{aligned}
    \end{equation}
    Take $\zeta=\eta(t)\partial _t (u^m)^h$ as the admissible test function in the weak formulation (4.2), where $\eta(t)\in C_0^\infty((0,+\infty))$ is a cut-off function satisfying $0\leq \eta(t)\leq 1$, $\eta(t)=1$ for $t\in [t_1,t_2]$. We have
    \begin{equation}
        \begin{aligned}
            \int_0^\infty\int_{\mathbb{H}^n}\eta \partial _t u^h \partial _t (u^m)^h\diff \mu(x)\diff t &=-\frac{1}{2}\int_0^\infty \eta(t)\partial _t\mathcal{E}_J((u^m)^h,(u^m)^h)\diff t\\
            &=\frac{1}{2}\int_0^\infty \partial_t \eta(t)\mathcal{E}_J((u^m)^h,(u^m)^h)\diff t. \nonumber
        \end{aligned}
    \end{equation}
    Notice that $(\delta^h u^m)(\delta^h u)\geq c(\delta^h u^{(m+1)/2})^2$, see \autocite{afp_dpa}, we obtain
    \begin{equation}
        \begin{aligned}
            \int_{t_1}^{t_2}\int_{\mathbb{H}^n} (\delta^h u^{(m+1)/2})^2\diff \mu(x)\diff t\leq C.\nonumber
        \end{aligned}
    \end{equation}
    Therefore, $\partial_t u^{(m+1)/2}\in L^2_{\text{loc}}(\mathbb{H}^n\times(0,\infty))$.
\end{proof}

\begin{proposition}
    The weak solutions to equation (1.1) are indeed strong solutions. Moreover, $\norm{\partial _t u(\cdot, t)}_1\leq \frac{2}{|m-1|t}\norm{u_0}_1$, for $m\neq 1$ for all $\tau>0$. 
\end{proposition}
\begin{proof}
    We begin with the case $m\neq 1$. From Lemma 5 and Lemma 6, $\partial _t u$ is a bounded radon measure and $\partial_t u^{(m+1)/2}\in L^2_{\text{loc}}(\mathbb{H}^n\times(0,\infty))$. Following Theorem 1.1 in \autocite{ssi_bp}, we obtain that $\partial _t u\in L^1_{\text{loc}}(\mathbb{H}^n\times(0,\infty))$. Furthermore, we get $\partial _t u\in L^{\infty}((\tau,\infty):L^1(\mathbb{H}^n))$ by estimate (4.1).
    
    If $m=1$, we obtain $\partial _t u\in L^1_{\text{loc}}(\mathbb{H}^n\times(0,\infty))$ from Lemma 6. By $L^1$ contractivity estimate, for every $t\geq \tau>0$, we get
    \begin{equation}
        \begin{aligned}
            \frac{ \norm{u(x,t+h)-u(x,t)}_1}{h}\leq \frac{ \norm{u(x,\tau+h)-u(x,\tau)}_1}{h}\leq c(\tau).\nonumber
        \end{aligned}
    \end{equation}
    Therefore, we prove $\partial _t u\in L^{\infty}((\tau,\infty):L^1(\mathbb{H}^n))$ in the linear case $m=1$.
\end{proof}

Now it's reasonable to use $u^m$ as a test function in the weak formulation of equation (1.1),
%since only need \partial _t u \times u^m,don't need \partial _t u^m. and $u^{m+1}$ can be approximated. mreover, if u is bounded, then u m+1 power is also reasonable. 
\begin{equation}
    \begin{aligned}
        \int_0^t \mathcal{E}_J(u^m,u^m) \diff s+\frac{1}{m+1}\int _{\mathbb{H}^n}|u|^{m+1}(x,t)\diff \mu(x)=\frac{1}{m+1}\int _{\mathbb{H}^n}|u_0|^{m+1}\diff \mu(x).\\
    \end{aligned}
\end{equation}
Then we have the following estimate:
\begin{proposition}
    If the initial data $u_0\in L^1(\mathbb{H}^n)\cap L^{\infty}(\mathbb{H}^n)$, then the strong solution to equation (1.1) satisfies
    \begin{equation}
        \begin{aligned}
            \int _0^\infty\mathcal{E}_J(u^m,u^m)\diff t\leq \frac{1}{m+1}\norm{u_0}_{m+1}^{m+1}.\\
        \end{aligned}
    \end{equation}
\end{proposition}

In fact, all the $L^p$ norms are non-increasing as time goes by. 
\begin{proposition}
    Any $L^p$ norm, $1\leq p \leq \infty$, of the strong solution to equation (1.1) with $u_0\in L^1(\mathbb{H}^n)\cap L^{\infty}(\mathbb{H}^n)$ is non-increasing in time. 
\end{proposition}

\begin{proof}
    The cases of $p=1$ and $p=\infty$ have been obtained from the contraction property in Proposition 3. 
    
    When $1<p<\infty$, we multiply the equation (1.1) by $u^{p-1}$. By Stroock-Varopoulos inequality, we have
    \begin{equation}
        \begin{aligned}
            \frac{\diff }{\diff t}\int _{\mathbb{H}^n}|u|^p(x,t)\diff \mu(x)&\leq -p\mathcal{E}_J(u^m,u^{p-1})\\
            &\leq -C\mathcal{E}_J(u^{\frac{m+p-1}{2}},u^{\frac{m+p-1}{2}})\leq 0.\\
        \end{aligned}
    \end{equation}
\end{proof}

\section{Existence and Uniqueness with General Initial Data}

In order to prove the existence of solutions corresponding to less restrictive initial data, we first prove the smoothing effects to obtain that the solutions are uniformly bounded away from $t=0$. These estimates will be used in the approximation process to obtain solutions with the general initial data. The smoothing effects are analogues of the results in \autocite{agf_dpa}.

\begin{proposition}
    Let $0<\alpha<2$, $m>0$, and take $p>p^*(m)=\max\{1,(1-m)Q/\alpha\}$. Then for every $u_0\in L^1(\mathbb{H}^n)\cap L^\infty(\mathbb{H}^n)$, the solution to the Problem (1.1) satisfies
    \begin{equation}
        \begin{aligned}
            \sup\limits_{x\in \mathbb{H}^n}|u(x,t)|\leq C(m,p,\alpha,Q, \Lambda)t^{-\gamma_p}\norm{u_0}_p^{\delta_p}
        \end{aligned}
    \end{equation}
    with $\gamma_p=(m-1+\alpha p/Q)^{-1}$ and $\delta_p=\alpha p \gamma_p/Q$.
\end{proposition}

\begin{proof}
    The parabolic Moser iterative technique will be used to prove the estimates. Fix $t$, consider the following sequence $t_k=(1-2^{-k})t$ with $0\leq k < +\infty$. Multiplying the equation (1.1) by $u^{p_k-1}$, then integrate in $\mathbb{H}^n\times[t_k,t_{k+1}]$. Here $p_k\geq p_0>1$ which will be determined afterwards. Using the Stroock-Varopoulos inequality and the decay of $L^p$ norms, we have
    \begin{equation}
        \begin{aligned}
            \norm{u(\cdot,t_k)}_{p_k}^{p_k}
            &\geq  C(Q, \alpha, \Lambda)\frac{4mp_k(p_k-1)}{(p_k+m-1)^2}\int _{t_k}^{t_{k+1}}\mathcal{E}_J(|u|^{\frac{p_k+m-1}{2}}(\cdot,\tau),|u|^{\frac{p_k+m-1}{2}}(\cdot,\tau))\diff \tau\\
            & \geq \frac{1}{d_k\norm{u(\cdot,t_k)}_{p_k}^{p_k}}\int _{t_k}^{t_{k+1}}\norm{u(\cdot,\tau)}_{p_k}^{p_k}\mathcal{E}_J(|u|^{\frac{p_k+m-1}{2}}(\cdot,\tau),|u|^{\frac{p_k+m-1}{2}}(\cdot,\tau))\diff \tau.\nonumber\\
        \end{aligned}
    \end{equation}
    Here $d_k=\frac{1}{C(Q, \alpha, \Lambda)}\cdot\frac{(p_k+m-1)^2}{4mp_k(p_k-1)}$.  By using the inequalities (2.4) and (2.5), we have
    \begin{equation}
        \begin{aligned}
            \int _{t_k}^{t_{k+1}}\norm{u(\cdot,\tau)}_{p_k}^{p_k}\mathcal{E}_J(|u|^{\frac{p_k+m-1}{2}}(\cdot,\tau),|u|^{\frac{p_k+m-1}{2}}(\cdot,\tau))\diff \tau&\geq C\int _{t_k}^{t_{k+1}}\norm{u(\cdot,\tau)}_{\frac{Q(2p_k+m-1)}{2Q-\alpha}}^{2p_k+m-1}\diff \tau\\
        &\geq C2^{-k}t\norm{u(\cdot,t_{k+1})}_{\frac{Q(2p_k+m-1)}{2Q-\alpha}}^{2p_k+m-1}.\nonumber\\
        \end{aligned}
    \end{equation}
    Therefore,
    \begin{equation}
        \begin{aligned}
            \norm{u(\cdot,t_{k+1})}_{p_{k+1}}\leq (C2^k d_k t^{-1})^{\frac{s}{2p_{k+1}}}\norm{u(\cdot,t_k)}_{p_k}^{\frac{sp_k}{p_{k+1}}},
        \end{aligned}
    \end{equation}
    where $p_{k+1}=s(p_k+\frac{m-1}{2})$ and $s=\frac{2Q}{2Q-\alpha}>1$. Hereafter, $\{p_k\}$ is an increasing sequence and $\lim\limits_{k\rightarrow+\infty}p_k = +\infty$, we discover that 
    \begin{equation*}
        \begin{aligned}
            p_k = A(s^{k} - 1) + p_0, ~~A = p_0 - \frac{(1-m)Q}{\alpha}.
        \end{aligned}
    \end{equation*}
    To guarantee that $A>0$, we impose the following condition on $p_0$:
    \begin{equation*}
        \begin{aligned}
            p_0 \geq \frac{(1 - m)Q}{\alpha}.
        \end{aligned}
    \end{equation*}
    Notice that if $m\geq 1$, we have
    \begin{equation*}
        \begin{aligned}
            d_k &= C(Q,\alpha,\Lambda) \frac{(p_k +m -1)^2}{mp_k^2} \cdot \frac{p_k}{p_k -1} \leq C(Q,\alpha,\Lambda)\frac{4(m-1)}{m}\cdot\frac{p_k}{p_k -1} \\
            &\leq C(Q,\alpha,\Lambda) \cdot \frac{p_k}{p_k -1} \leq C(Q,\alpha,\Lambda) \frac{p_k}{p_0-1}.
        \end{aligned}
    \end{equation*}
    If $m< 1$, we have
    \begin{equation*}
        \begin{aligned}
            d_k &= C(Q,\alpha,\Lambda) \frac{(p_k +m -1)^2}{mp_k^2} \cdot \frac{p_k}{p_k -1} \leq C(Q,\alpha,\Lambda) \frac{\alpha m}{(1-m)Q-\alpha}\cdot p_k.
        \end{aligned}
    \end{equation*}
    Define $U_k = \norm{u(\cdot, t_k)}_{p_k}$. From the iterative inequalities (5.2), we have 
    \begin{equation*}
        \begin{aligned}
            U_{k+1} \leq C^{\frac{k}{p_{k+1}}}t^{-\frac{s}{2p_{k+1}}}U_k^{\frac{sp_k}{p_{k+1}}}.
        \end{aligned}
    \end{equation*}
    Therefore, we obtain
    \begin{equation*}
        \begin{aligned}
            U_k \leq C^{\alpha_k}t^{-\beta_k}U_0^{\sigma_k}.
        \end{aligned}
    \end{equation*}
    Here
    \begin{equation*}
        \begin{aligned}
            \alpha_k = \frac{1}{p_k}\sum_{j=1}^{k-1}(k-j)s^j,~~\beta_k = \frac{1}{2p_k} \sum_{j=1}^{k}s^j,~~\sigma_k = \frac{s^k p_0}{p_k}.
        \end{aligned}
    \end{equation*}
    Let $k$ approach infinity, by the iteration process, we have
    \begin{equation*}
        \begin{aligned}
            \lim\limits_{k\rightarrow+\infty}\alpha_k = \frac{Q(Q-\alpha)}{\alpha^2 A}, ~~\lim\limits_{k\rightarrow+\infty}\beta_k = \frac{Q}{\alpha A},~~\lim\limits_{k\rightarrow+\infty}\sigma_k = \frac{p_0}{ A}.
        \end{aligned}
    \end{equation*}
    Hereafter,
    \begin{equation*}
        \begin{aligned}
            \norm{u(\cdot,t)}_{\infty} & = \lim\limits_{k\rightarrow +\infty} \norm{u(\cdot, \frac{t}{1-2^{-k}})}_{p_k}\\
            &\leq \lim\limits_{k\rightarrow +\infty} C(Q,\alpha,\Lambda)(\frac{t}{1-2^{-k}})^{-\frac{Q}{\alpha A}} U_0^{\frac{p_0}{A}}\\
            &\leq C(Q,\alpha, \Lambda) t^{-\frac{Q}{(m-1)Q+p_0\alpha}}\norm{f}_{p_0}^{\frac{p_0\alpha}{(m-1)Q+p_0\alpha}}.
        \end{aligned}
    \end{equation*}
    By the above Moser iteration technique, we obtain the desired smoothing effects estimates (5.1). 
\end{proof}

\begin{remark}
    If $m\leq m^*$, the constant $A$ goes to zero when $p\rightarrow p*(m)$. Hereafter, the constants in the estimates (5.1) blow up as $p\rightarrow p*(m)$. If $m> m^*$ and $p\rightarrow 1$, the $\frac{1}{p-1}$ of constant $d_k$ in Proposition 9 blow up. 

    However, an iterative interpolation argument, see the analogues in \cite{agf_dpa}, will give rise to the $L^1-L^\infty$ smoothing effect as $m>m_*$.
\end{remark}

\begin{proposition}
    Let $0<\alpha<2$, $m>m_*$. For every $u_0\in L^1(\mathbb{H}^n)\cap L^\infty(\mathbb{H}^n)$, the solution to equation (1.1) satisfies
    \begin{equation}
        \begin{aligned}
            \sup\limits_{x\in \mathbb{H}^n}|u(x,t)|\leq C(m,\alpha,Q)t^{-\gamma_1}\norm{u_0}_1^{\delta_1}
        \end{aligned}
    \end{equation}
    with $\gamma_1=(m-1+\alpha/Q)^{-1}$ and $\delta_1=\alpha\gamma_1/Q$.
\end{proposition}

\begin{proof}
    For fixed $t$, let $\tau_k = 2^{-k}t$. Here $0\leq k<+\infty$. Since $p^*(m)< 1$ if $m>m^*$, Applying the smoothing effects estimates (5.1) in Proposition 9 on the interval $[\tau_1, \tau_0]$ with $p=2$, we have
    \begin{equation*}
        \begin{aligned}
            \norm{u(\cdot,t)}_{\infty} \leq C(Q,\alpha,\Lambda) (\frac{t}{2})^{-\gamma_2}\norm{u(\cdot,\tau_1)}_1^{\frac{\alpha\gamma_2}{Q}}\norm{u(\cdot,\tau_1)}_{\infty}^{\frac{\alpha\gamma_2}{Q}}.
        \end{aligned}
    \end{equation*}
    Considering the iterative sequence $\tau_k$, applying the estimates (5.1) with $p=2$ on the interval $[\tau_2, \tau_1]$, we have
    \begin{equation*}
        \begin{aligned}
            \norm{u(\cdot,t)}_{\infty} \leq C(Q,\alpha,\Lambda) (\frac{t}{2})^{-\gamma_2}\norm{u(\cdot,\tau_1)}_1^{\frac{\alpha\gamma_2}{Q}}\cdot(C(Q,\alpha,\Lambda) (\frac{t}{4})^{-\gamma_2}\norm{u(\cdot,\tau_2)}_2^{\frac{2\alpha\gamma_2}{Q}})^{\frac{\alpha\gamma_2}{Q}}
        \end{aligned}
    \end{equation*}
    Following the iterative process, since $\norm{u(\cdot, t)}_1\leq \norm{u(\cdot, 0)}_1$ for each $t\geq 0$, we obtain
    \begin{equation*}
        \begin{aligned}
            \norm{u(\cdot,t)}_{\infty} \leq  C(Q,\alpha,\Lambda)^{a_k}\cdot 2^{b_k}\cdot t^{-c_k}\cdot \norm{u_0}_1^{d_k}\cdot \norm{u(\cdot,\tau_k)}_2^{e_k}.
        \end{aligned}
    \end{equation*}
    Since $m> m^*$, the constant $\frac{\gamma_2\alpha}{Q}<1$. Then,
    \begin{equation*}
        \begin{aligned}
            &\lim\limits_{k\rightarrow+\infty}a_k = \lim\limits_{k\rightarrow+\infty}\sum_{j=0}^{k-1}(\frac{\gamma_2\alpha}{Q})^j = \frac{(m-1)Q+2\alpha}{(m-1)Q+\alpha},\\
            &\lim\limits_{k\rightarrow+\infty}b_k = \lim\limits_{k\rightarrow+\infty}\sum_{j=0}^{k-1}\gamma_2(j+1)(\frac{\gamma_2\alpha}{Q})^j = \frac{(m-1)Q+2\alpha}{((m-1)Q+\alpha)^2},\\
            &\lim\limits_{k\rightarrow+\infty}c_k = \lim\limits_{k\rightarrow+\infty}\sum_{j=0}^{k-1}\gamma_2(\frac{\gamma_2\alpha}{Q})^j = \frac{Q}{(m-1)Q+\alpha} = \gamma_1,\\
            &\lim\limits_{k\rightarrow+\infty}d_k = \lim\limits_{k\rightarrow+\infty}\sum_{j=1}^{k-1}(\frac{\gamma_2\alpha}{Q})^j = \frac{\alpha}{(m-1)Q+\alpha} = \gamma_1\cdot\frac{\alpha}{Q},\\
            &\lim\limits_{k\rightarrow+\infty}e_k = \lim\limits_{k\rightarrow+\infty}2(\frac{\gamma_2\alpha}{Q})^k = 0.
        \end{aligned}
    \end{equation*}
    which are similar to the constants in \autocite{agf_dpa}. Using the fact that $\norm{u(\cdot,\tau_k)}_2\leq \norm{u_0}_2$, we prove
    \begin{equation*}
        \begin{aligned}
            \sup\limits_{x\in \mathbb{H}^n}|u(x,t)|\leq C(m,\alpha,Q)t^{-\gamma_1}\norm{u_0}_1^{\frac{\gamma_1\alpha}{Q}}
        \end{aligned}
    \end{equation*}
\end{proof}

A generalization of the existence of the weak solutions to the less restrictive initial data, indeed which are strong with respect to time $t$ variable, will be proved at the end of this section. 

If $m> m^*$, constructing the approximating solutions corresponding to the approximating initial data $\{u_{0,k}\} \in L^1(\mathbb{H}^n)\cap L^{\infty}(\mathbb{H}^n)$. Since we have already obtained the existence of the weak solutions to the bounded initial data, the point here is to employ the smoothing effects estimates (5.3) for $m>m^*$ to give a uniform $L^\infty(\mathbb{H}^n)$ norm of the approximating solutions away $t=0$. Hereafter, passing the weak formualtions to the limit, we prove the weak formulation holds for the general inital data $u_0\in L^1(\mathbb{H}^n)$. 

If $m\leq m^*$, the smoothing effects estimates  (5.2) only hold for $p>\max\{1,(1-m)Q/\alpha\}$. Therefore, we obtain the existence of the weak solutions for each $u_0\in L^1(\mathbb{H}^n)\cap L^p(\mathbb{H}^n)$ with suitable parameter $p$.

\begin{theorem}
    Let $0<\alpha<2$, $m>0$. Assume $u_0\in L^1(\mathbb{H}^n)$ if $m>m^*$; $u_0\in L^1(\mathbb{H}^n)\cap L^p(\mathbb{H}^n)$ with $p>\max\{1,(1-m)Q/\alpha\}$ if $m\leq m^*$, there exists a strong solution to equation (1.1) with the initial data $u_0$. 
\end{theorem}

\begin{proof}
    Let $\{u_{0,k}\}\subset L^1(\mathbb{H}^n)\cap L^\infty(\mathbb{H}^n)$ be a sequence of functions converging to $u_0$ in $L^1(\mathbb{H}^n)\cap L^p(\mathbb{H}^n)$ satisfying 
    \begin{equation*}
        \begin{aligned}
            \norm{u_{0,k}}_p\leq \norm{u_0}_p.
        \end{aligned}
    \end{equation*}
    Denote by $\{u_k\}$ the sequence of the solutions corresponding to the initial data $u_{0,k}$. 

    \noindent Applying the $L^1$-contraction property and Crandall-Ligget's theorem, we obtain $u_k\rightarrow u$ in $C([0,\infty):L^1(\mathbb{H}^n))$. By the smoothing effects estimates (5.2) or (5.3), we have 
    \begin{equation*}
        \begin{aligned}
            \sup\limits_{x\in \mathbb{H}^n}|u_k(x,t)|\leq C(m,p,\alpha,Q, \Lambda)t^{-\gamma_p}\norm{u_0}_p^{\delta_p}\leq C(m,p,\alpha,Q,\Lambda)t^{-\gamma_p}.
        \end{aligned}
    \end{equation*}

    \noindent Since $\{u_k\}$ are strong solutions satisfying (3.10), the weak formulations implies
    \begin{equation*}
        \begin{aligned}
            \int_{t_1}^{t_2}\mathcal{E}_J(u_k^m,u_k^m)\diff t + \frac{1}{m+1} \int_{\mathbb{H}^n}|u_k|^{m+1}(x,t_2)\diff \mu(x) = \frac{1}{m+1}\int_{\mathbb{H}^n}|u_k|^{m+1}(x,t_1)\diff \mu(x).
        \end{aligned}
    \end{equation*}
    Hereafter, the norm ${ L^2((\tau,\infty):\dot{\mathcal{H}_{\mathcal{L}}}(\mathbb{H}^n))}$ of the solutions $\{u_k\}$ are uniformly bounded which are independent of $k$. By the Banach-Alaoglu theorem, $\{u_k\}$ converge to $u$ in the weak-$^*$ topology on 
    $L^2_{\rm{loc}}((0,\infty):\dot{\mathcal{H}_{\mathcal{L}}}(\mathbb{H}^n))$. Notice that the following weak formulations of the solutions $\{u_k\}$ hold,
    \begin{equation*}
        \begin{aligned}
            \int _0^\infty \int _{\mathbb{H}^n} u_k\partial _t \zeta-\int_0^\infty \mathcal{E}_J(u_k^m,\zeta)=0,~~ \rm{for ~each} ~\zeta\in C^2_{c}(\mathbb{H}^n\times(0,\infty)).
        \end{aligned}
    \end{equation*}
    Passing to the limit, we obtain
    \begin{equation*}
        \begin{aligned}
            \int _0^\infty \int _{\mathbb{H}^n} u\partial _t \zeta-\int_0^\infty \mathcal{E}_J(u^m,\zeta)=0,~~ \rm{for ~each} ~\zeta\in C^2_{c}(\mathbb{H}^n\times(0,\infty)).
        \end{aligned}
    \end{equation*}
    Furthermore, we have
    \begin{equation}
        \begin{aligned}
            \int _{\mathbb{H}^n}|u(x,t)-u_0(x)|\diff \mu(x)&\leq \int _{\mathbb{H}^n}|u(x,t)-u_k(x,t)|\diff \mu(x)+\int_{\mathbb{H}^n}|u_k(x,t)-u_{0,k}(x)|\diff \mu(x)\\
            &+\int _{\mathbb{H}^n}|u_{0,k}(x)-u_0(x)|\diff\mu(x).\nonumber\\
        \end{aligned}
    \end{equation}
    Passing to the limit, we derive the initial condition.
    \begin{equation*}
        \begin{aligned}
            u(x,0) = u_0(x) ~~\rm{almost~everywhere}.
        \end{aligned}
    \end{equation*}
    The above constructed weak solution $u$ is strong satisfying (3.10) by Proposition 6.
\end{proof}

\section{Further Properties}
\subsection{Conservation of mass if $m>m^*$}
\begin{proposition}
    Let $m>m^*$, $0<\alpha<2$. $J$ satisfies assumption (1.3). For each $u_0\in L^1(\mathbb{H}^n)$, the corresponding solution $u$ satisfies the property, conservation of mass: $\int _{\mathbb{H}^n}u(x,t)\diff \mu(x)=\int_{\mathbb{H}^n}u_0(x,t)\diff \mu(x)$.
\end{proposition}

\begin{proof}
    First we consider $u_0\in L^1(\mathbb{H}^n)\cap L^\infty(\mathbb{H}^n)$, then the corresponding solution $u\in L^1(\mathbb{H}^n)\cap L^\infty(\mathbb{H}^n)$. Choosing a non-negative and nonincreasing cutoff function $\psi(s)$ such that $\psi(s)=1$ for $0\leq s \leq 1$, $\psi(s)=0$ for $s\geq 2$. Define $\psi_R(x)=\psi(|x|/R)$. Employing $\psi_R$ as a test function in the weak formulation, we have
    $$\int _{\mathbb{H}^n}u(x,t)\psi_R(x)\diff \mu(x)-\int _{\mathbb{H}^n}u(x,0)\psi_R(x)\diff \mu(x)=-\int_0^t\int_{\mathbb{H}^n}u^m\mathcal{L}\psi_R\diff\mu(x)\diff t.$$
    By the scaling property of test functions, we have  
    \begin{equation*}
        \begin{aligned}
            \norm{\mathcal{L}\psi_R}_q\leq C R^{-\alpha+\frac{Q}{q}},~~\rm{for~}q\geq 1.
        \end{aligned}
    \end{equation*}
    Applying H\"older's inequality, we have
    $$|\int _{\mathbb{H}^n}u(x,t)\psi_R(x)\diff \mu(x)-\int _{\mathbb{H}^n}u(x,0)\psi_R(x)\diff \mu(x)|\leq C t R^{-\alpha+\frac{Q(p-1)}{p}}.$$ 
    Here $p=\max\{1,\frac{1}{m}\}$. 
    %here the choice of p is u^m not bounded, u^1 boudned , using Holder ,  we need u^m^{1/m} to get u^1.
    Since $m>m^*$, the exponent of $R$ is negative. Letting $R$ go to infinity, we derive the conservation of mass property.
    
    \noindent If $u_0\in L^1(\mathbb{H}^n)$, let $\{u_{0,k}\}\subset L^1(\mathbb{H}^n)\cap L^\infty(\mathbb{H}^n)$ be a sequence of functions converging to $u_0$ in $L^1(\mathbb{H}^n)$. Denote by $\{u_k\}$ the sequence of the solutions corresponding to the initial data $\{u_{0,k}\}$. The conservation of mass for bounded initial data gives us
    \begin{equation}
        \begin{aligned}
            \int _{\mathbb{H}^n}u_k(x,t)\diff \mu=\int_{\mathbb{H}^n}u_{0,k}(x,t)\diff \mu.\nonumber 
        \end{aligned}
    \end{equation}
    Passing to the limit, we get
    \begin{equation}
        \begin{aligned}
            \int _{\mathbb{H}^n}u(x,t)\diff \mu=\int_{\mathbb{H}^n}u_0(x,t)\diff \mu ~~\text{for} ~\text{each}~u_0\in L^1(\mathbb{H}^n).\nonumber 
        \end{aligned}
    \end{equation}
\end{proof}

\subsection{Conservation of mass if $m=m^*$}
If $m= m^*$, the exponent $-\alpha+\frac{Q(p-1)}{p}$ in the proof of Proposition 11 is zero and we can not obtain the conservation of mass by the above  proof. In order to overcome this difficulty, we split the integral $\frac{\diff }{\diff t}\int_{\mathbb{H}^n}u(x,t)\psi_R(x)\diff \mu(x)$ into two parts and estimate each part separately.
 
\begin{proposition}
    Let $m=m^*$, $0<\alpha<2$. $J$ satisfies assumption (1.3). For each $u_0\in L^1(\mathbb{H}^n)\cap L^p(\mathbb{H}^n)$ with $p>1$, the solution $u$ corresponding to the initial data $u_0$ satisfies the property, conservation of mass: $\int _{\mathbb{H}^n}u(x,t)\diff \mu=\int_{\mathbb{H}^n}u_0(x,t)\diff \mu$.
 \end{proposition}

\begin{proof}
    For each given $\delta>0$, choose $R_0$ such that $\int _{\mathbb{H}^n}|u \cdot \chi_{\{|x|\geq R_0\}}|\leq \delta$. Then $u$ can be rewritten as $u=u_1+u_2$ with
    \begin{equation}
        \begin{aligned}
            u_1=u\cdot \chi_{\{|x|\leq R_0\}}, ~~u_2=u\cdot \chi_{\{|x|\geq R_0\}}.\nonumber
        \end{aligned}
    \end{equation}
    Observe that $u^m=u_1^m+u_2^m$. We begin with the bounded solution $u$ corresponding to the bounded initial data $u_0\in L^1(\mathbb{H}^n)\cap L^\infty(\mathbb{H}^n)$. 
    
    \noindent Multiplying the same test function $\psi_R(\cdot)$ as the proof of Proposition 11, we have
    \begin{equation}
        \begin{aligned}
            |\frac{\diff }{\diff t}\int_{\mathbb{H}^n}u(x,t)\psi_R(x)\diff \mu(x)|&=|\int _{\mathbb{H}^n}u_1^m\mathcal{L}\psi_R\diff\mu(x)+\int _{\mathbb{H}^n}u_2^m\mathcal{L}\psi_R\diff\mu(x)|\\
            &\leq |\int _{\mathbb{H}^n}u_1^m\mathcal{L}\psi_R\diff\mu(x)|+|\int _{\mathbb{H}^n}u_2^m\mathcal{L}\psi_R\diff\mu(x)|\\
            &:=I_1+I_2.\nonumber \\
        \end{aligned}
    \end{equation}
    By H\"older's inequality, we have
    \begin{equation}
        \begin{aligned}
            \norm{u_1^m}_r\leq \norm{u_1}_1^{m}|\{|x|\leq R_0\}|^{\frac{1-mr}{r}}\leq CR_0^{\frac{Q(1-mr)}{r}}~~\text{for}~\text{every}~1\leq r<1/m.\nonumber\\
        \end{aligned}
    \end{equation}
    Hereafter,
    \begin{equation}
        \begin{aligned}
            I_1\leq \norm{u_1^m}_r\norm{\mathcal{L}\psi_R}_{\frac{r}{r-1}}\leq CR_0^{\frac{Q(1-mr)}{r}}R^{-\alpha+\frac{Q(r-1)}{r}}.\nonumber
        \end{aligned}
    \end{equation}
    Observe that $\frac{Q(1-mr)}{r}=\alpha-\frac{Q(r-1)}{r}$ when $m=m^*=\frac{Q-\alpha}{Q}$, we obtain
    \begin{equation}
        \begin{aligned}
            I_1\leq C(R_0/R)^{\frac{Q(1-mr)}{r}}.\nonumber\\
        \end{aligned}
    \end{equation}
    % We can choose $r$ such that $1\leq r<{1}/{m}$, $\frac{Q(1-mr)}{r}>0$.
    For fixed $\delta$ and $R_0$, $I_1$ goes to zero by letting $R$ go to infinity. Furthermore, 
    \begin{equation*}
        \begin{aligned}
            \norm{u_2^m}_{\frac{1}{m}}=\norm{u_2}_1^m\leq \delta^m.
        \end{aligned}
    \end{equation*}
    Then we have,
    \begin{equation}
        \begin{aligned}
            I_2\leq \norm{u_2^m}_{\frac{1}{m}}\norm{\mathcal{L}\psi_R}_{\frac{1}{1-m}}\leq C\delta^m R^{0}=C\delta^m.\nonumber
        \end{aligned}
    \end{equation}
    %\delta is fixed now, so we should let R goes to infinity first, in I_1+I_2, u_1 and u_2 have already been determined.
    Letting $R\rightarrow \infty$, we obtain $I_1+I_2\leq C\delta^m$. We conclude that 
    \begin{equation*}
        \begin{aligned}
            |\frac{\diff }{\diff t}\int_{\mathbb{H}^n}u(x,t)\psi_R(x)\diff \mu(x)|\rightarrow 0.
        \end{aligned}
    \end{equation*}
    by letting $\delta\rightarrow 0$. We obtain the conservation of mass with $u_0\in L^1(\mathbb{H}^n)\cap L^\infty(\mathbb{H}^n)$ if $m=m^*$. 

    \noindent If $u_0\in L^1(\mathbb{H}^n)\cap L^p(\mathbb{H}^n)$ with arbitrary $p>1$, let $\{u_{0,k}\}\subset L^1(\mathbb{H}^n)\cap L^\infty(\mathbb{H}^n)$ be a sequence of functions converging to $u_0$ in $L^1(\mathbb{H}^n)\cap L^p(\mathbb{H}^n)$. Denote by $\{u_k\}$ the sequence of the solutions corresponding to the initial data $\{u_{0,k}\}$. The conservation of mass for bounded initial data gives us
    \begin{equation}
        \begin{aligned}
            \int _{\mathbb{H}^n}u_k(x,t)\diff \mu=\int_{\mathbb{H}^n}u_{0,k}(x,t)\diff \mu.\nonumber 
        \end{aligned}
    \end{equation}
    Passing to the limit, we obtain
    \begin{equation}
        \begin{aligned}
            \int _{\mathbb{H}^n}u(x,t)\diff \mu=\int_{\mathbb{H}^n}u_0(x,t)\diff \mu ~~\text{for} ~\text{each}~u_0\in L^1(\mathbb{H}^n)\cap L^p(\mathbb{H}^n).\nonumber 
        \end{aligned}
    \end{equation}
    % Passing to the limit, this property holds for the initial data $u_0\in L^1(\mathbb{H}^n)\cap L^p(\mathbb{H}^n)$ with $p>1$.
 \end{proof}
 
 \subsection{Extinction if $m<m^*$}
 We have already proved that the mass of the solution does not change as $t\rightarrow \infty$ if $m\geq m^*$. On the other hand, for $m<m^*$, there is a finite extinction time $T>0$ such that $u(\cdot,T)\equiv0$ almost everywhere in $\mathbb{H}^n$, see an analogous result in \autocite{agf_dpa}.

\begin{proposition}
    Let $m<m^*$, $0<\alpha<2$, $J$ satisfies assumption (1.3). For every $u_0\in L^1(\mathbb{H}^n)\cap L^p(\mathbb{H}^n)$ with $p>\max\{1,(1-m)Q/\alpha\}$, there is a finite time $T>0$ of the corresponding solution $u$ such that $u(\cdot,T)\equiv0$ almost everywhere in $\mathbb{H}^n$.
\end{proposition}

\begin{proof}
    From inequality (2.3), (2.5) and (4.5), we have
    \begin{equation}
        \begin{aligned}
            \frac{\diff }{\diff t}\int _{\mathbb{H}^n}|u|^p(x,t)\diff \mu(x)&\leq-p\mathcal{E}_J(u^m,u^{p-1})\\
            &\leq -C\mathcal{E}_J(u^{\frac{m+p-1}{2}},u^{\frac{m+p-1}{2}})\\
            &\leq -C\norm{u^{\frac{m+p-1}{2}}}_{\frac{2Q}{Q-\alpha}}^2=-C(\int_{\mathbb{H}^n}|u|^{\frac{(p+m-1)Q}{Q-\alpha}}\diff \mu)^{\frac{Q-\alpha}{Q}}.\nonumber
        \end{aligned}
    \end{equation}
    %p=(1-m)Q/alpha, since p>1, p can not be 1, that's the reason why m<m^*
    Observe that $p=\frac{(p+m-1)Q}{Q-\alpha}$ if $p=\frac{(1-m)Q}{\alpha}$. Define
    \begin{equation*}
        \begin{aligned}
            J(t)=\norm{u(\cdot,t)}_p^p,
        \end{aligned}
    \end{equation*}
    Then,
    \begin{equation}
        \begin{aligned}
            J'(t)+C J^{\frac{Q-\alpha}{Q}}(t)\leq 0. \nonumber\\
        \end{aligned}
    \end{equation}
    Hereafter, we derive the extinction of the solution in finite time since $J(0)$ is finite. 
 \end{proof}

\section{C$^\alpha$ regularity}
In this section, we prove continuity results of the constructed solutions to the equation (1.1). Instead, the equation (1.1) can be rewritten as 
\begin{equation*}
    \partial_t u^{1/m}+\mathcal{L}u=0.
\end{equation*}
%part (e)(f) of Theorem 2 and Theorem 3. In previous sections, we have proved the unique weak bounded solutions exist when $\beta $ is continuous and increasing. Moreover, these weak solutions are bounded with a initial value $\beta(u(\cdot,0))\in L^1(\mathbb{H}^n)\cap L^\infty(\mathbb{H}^n)$.Considering the bounded weak solution, we can adjust $\beta(s)$ when $s\geq \norm{u}_\infty$. It's reasonable to assume $\beta'$ is bounded over $\mathbb{R}.$
In order to prove the continuity results, we will construct the iterated sequence of solutions with iterated nonlinearity functions in decreasing space-time cylinders. For instance, we will construct the following iterative functions:
\begin{equation*}
    \begin{aligned}
        u_k(x,t) = \frac{u(R^{-k}x, c_0 R^{-\alpha k}t) - \mu_k}{\omega_k}.
    \end{aligned}
\end{equation*}
Here the constants $R$, $\mu_k$, and $\omega_k$ will be chosen in the subsequent proof. Moreover, if the sequence of functions $\{u_k\}$ satisfies the specific consistent regularity,  we can obtain the Hölder modulus continuity of the solution $u$. 

Observe that the function $u_k$ is a solution to the following equation with the nonlinearity function $\beta_k(s)$,
\begin{equation*}
    \begin{aligned}
        \partial_t \beta_k(u)+\mathcal{L}u=0,
    \end{aligned}
\end{equation*}
 
\begin{equation*}
    \begin{aligned}
        \beta_k (s) = \frac{|\omega_k s + \mu_k|^{\frac{1}{m}-1}(\omega_k s + \mu_k)}{\omega_k \cdot c_0}.
    \end{aligned}
\end{equation*}

This naturally leads to the investigations of the following equations with the general nonlinearity functions $\beta(s)$:
\begin{equation}
    \begin{aligned}
        \partial_t \beta(u)+\mathcal{L}u=0,
    \end{aligned}
\end{equation}
Here,
\begin{equation*}
    \beta(s)=a|bs+c|^{\frac{1}{m}-1}(bs+c) ~~~\text{for}~ \text{some} ~a,b>0, ~s\in \mathbb{R}.
\end{equation*}

We begin with an energy inequality about the weak solution to the equation (7.1) by the weak formulation. First of all, we use the following equalities frequently on the Heisenberg group $\mathbb{H}^n$. These equalities are mentioned in Corollary 1.6 \autocite{sea_fg} using the polar coordinates on the nilpotent groups. 

\begin{equation}
    \int_{a\leq |x|\leq b}|x|^{\alpha-Q}\diff x=
        \left\{
            \begin{array}{ll}
                C_0\alpha^{-1}(b^{\alpha}-a^{\alpha})~~ &\text{if}~\alpha\neq0,  \\
                C_0\log(b/a)~~&\text{if} ~\alpha=0.\\
            \end{array}
        \right.
\end{equation}
Here $a\in \mathbb{C}$ and $0<a<b<\infty$.  

To obtain the energy inequality, we regard $\zeta=(u-\psi)_+$ as a test function in the weak formulation. Here $\psi$ is a nonnegative Lipschitz barrier function satisfying,
\begin{equation}
    \begin{aligned}
        \int _{|x^{-1} \cdot y|>1}|\psi(x)-\psi(y)|J(x,y)\diff y<C<\infty ~~\text{for} ~\text{every}~x\in \mathbb{H}^n.
    \end{aligned}
\end{equation}  

Recall that we have obtained $\partial _t \beta(u) \in L^\infty((\tau,\infty): L^1(\mathbb{H}^n))$ in Section 4. Hereafter, the function $\zeta$ is an admissible test function. Define the functional,
\begin{equation}
    \begin{aligned}
        \mathcal{B}_{\psi}(u)=\int _0^{(u-\psi)_+}\beta'(s+\psi)s\diff s.
    \end{aligned}
\end{equation}
Therefore, 
\begin{equation}
    \begin{aligned}
        \int _{\mathbb{H}^n}\mathcal{B}_\psi(u(x,t))\diff x|^{t_2}_{t_1}+\int _{t_1}^{t_2}\mathcal{E}_{J}(u,(u-\psi)_+)(t)\diff t=0.
    \end{aligned}
\end{equation}

%%%***>>>>>>> Here is the important behind scenes:
% \begin{itemize}
%     \item First of all, in Theorem 6, we have to prove u is bounded in $\dot{\mathcal{H}_L}$, therefore, we have to obtain the uniform bound of $u_n$ of $\dot{\mathcal{H}_L}$. Since we have to consider the case $u_n \psi ---> u_n \psi$.
%     \item in other cases, if we have already a solution u,  the problem is to regard u or other things as the test functions, then we only need this u or other functions, they are $L^\infty$ and $\dot{\mathcal{H}_L}$, then by banach alaoglu theorem, we can pass the convergence
%     \item here, we have $\partial_t \beta(u)$ has derivative, therefore, away the singularity , u is also hasing detivative, u can be considered a composition function of $\partial_t \beta(u)$ away the singularity points. In this case, we can split $\partial_t \beta(u) = \beta'(s) \partial_t u$. And then we use the improper integral. Meanwhile, we use the approximated sequence of the $\zeta$ as the test function, we know this function is $L^\infty$ and $\dot{\mathcal{H}_L}$, which makes sense. Moreover, computing the improper integral's limit:
%     \item 
%     \begin{equation*}
%         \begin{aligned}
%             \int{\beta'(s)}\partial_t  u \times boundedfunction \rightrightarrows/ \leq C \int \beta'( u) = C \beta(u) which ~ is ~finite.
%         \end{aligned}
%     \end{equation*}
% \end{itemize}

In order to obtain a useful inverse Sobolev energy inequality, we estabish the next Lemma first to give an estimate of the functional $\mathcal{B}_{\psi}(u)$.
\begin{lemma}
    If $l=\inf \psi \geq 0$ and $M=\sup_{\{u\geq \psi\}} u<\infty$, the following inequality holds
        \begin{equation}
            \begin{aligned}
                \Lambda_1 (u-\psi)_+^{2}\leq \mathcal{B}_{\psi}(u) \leq \Lambda_2(u-\psi)_+.
            \end{aligned}
        \end{equation}
    Here $\Lambda_1=\frac{1}{2}\inf\limits_{l\leq s\leq M}\beta'(s)$ and $\Lambda_2=\beta(M)-\beta(l).$
\end{lemma}
  
\begin{proof}
    First, we give the lower bound of $\mathcal{B}_{\psi}(u)$,
    \begin{equation*}
        \begin{aligned}
            \mathcal{B}_{\psi}(u) \geq 2 \Lambda_1 \int_0^{(u-\psi)_+}s\diff s = \Lambda_1(u-\psi)_+^2.
        \end{aligned}
    \end{equation*}
    Since $\beta'(s)\geq 0$ for all $s\in \mathbb{R}$, we have
    \begin{equation*}
        \begin{aligned}
            \mathcal{B}_{\psi}(u) \leq (u-\psi)_+ \cdot \int_{0}^{(u-\psi)_+} \beta'(s+\psi)\diff s\leq \Lambda_2(u-\psi)_+.
        \end{aligned}
    \end{equation*}
\end{proof}
  
Now are ready to prove the inverse Sobolev energy inequality, see the similar results on $\mathbb{R}^n$ in \autocites{rtfp_cl}{nfe_dpa}. Recall that $\mathcal{E}_J(f,f)^{\frac{1}{2}}$ is equivalent to the homogeneous fractional Sobolev norm $\norm{\mathscr{L}^{\alpha/4}f}_2$, we obtain the following energy estimates,
  
\begin{lemma}
    For any nonnegative Lipschitz barrier function $\psi$ satisfying (7.3). If $u$ is a weak solution to (7.1) satisfying $l=\inf\psi\geq 0$ and $M=\sup_{\{u\geq \psi\}} u<\infty$, then
    \begin{equation}
        \begin{aligned}
            \Lambda_1\norm{(u-\psi)_+(\cdot,t_2)}_{2}^{2}+ C_1&\int _{t_1}^{t_2}\norm {\mathscr{L}^{\alpha/4}((u-\psi)_+)(\cdot,t)}_2^2\diff t\\
            &\leq \Lambda_2\norm{(u-\psi)_+(\cdot,t_1)}_1+C_2\int _{t_1}^{t_2}z(t)\diff t,
        \end{aligned}
    \end{equation}
    Where $z(t)=\norm{(u-\psi)_+(\cdot,t)}_1+\norm{\mathbbm{1}_{\{u(\cdot,t)>\psi(\cdot)\}}}_1$. Notice that $C_1$ and $C_2$ are positive numbers which are independent of the choice of $\beta$. 
\end{lemma}

\begin{proof}
    From equation (7.5), we have
    \begin{equation*}
        \int _{\mathbb{H}^n}\mathcal{B}_\psi(u(x,t_2))\diff \mu(x)+\int _{t_1}^{t_2}\mathcal{E}_{J}(u,(u-\psi)_+)(t)\diff t = \int _{\mathbb{H}^n}\mathcal{B}_\psi(u(x,t_1))\diff x.
    \end{equation*}
    To estimate the term $\int _{t_1}^{t_2}\mathcal{E}_{J}(u,(u-\psi)_+)(t)\diff t$, we split it into three parts,
    \begin{equation*}
        \begin{aligned}
            \int _{t_1}^{t_2}\mathcal{E}_{J}(u,(u-\psi)_+)(t)\diff t  = & \int _{t_1}^{t_2}\mathcal{E}_{J}((u-\psi)_+,(u-\psi)_+)(t)\diff t \\
            &+ \int _{t_1}^{t_2}\mathcal{E}_{J}((u-\psi)_-,(u-\psi)_+)(t)\diff t \\
            & + \int _{t_1}^{t_2}\mathcal{E}_{J}(\psi,(u-\psi)_+)(t)\diff t : = I_1 + I_2 +I_3   
        \end{aligned}
    \end{equation*}
    Here $I_2\geq 0$. Furthermore, we have the following estimate of $\mathcal{E}_{J}(\psi,(u-\psi)_+)(t)$,
    \begin{equation*}
        \begin{aligned}
            \mathcal{E}_{J}(\psi,(u-\psi)_+)(t)  &=  \int_{\mathbb{H}^n} \int_{|x^{-1}\cdot y|\geq 1}[\psi(x)-\psi(y)][(u-\psi)_+(x)-(u-\psi)_+(y)]J(x,y)\\
            &~~~+ \int_{\mathbb{H}^n} \int_{|x^{-1}\cdot y|\leq 1}[\psi(x)-\psi(y)][(u-\psi)_+(x)-(u-\psi)_+(y)]J(x,y)\\
            &:= J_1 + J_2
        \end{aligned}
    \end{equation*}
    Since the function $\psi$ satisfies the condition (7.3), we have
    \begin{equation*}
        |J_1| \leq C \int_{\mathbb{H}^n} (u-\psi)_+(y)\diff \mu(y)
    \end{equation*}
    Notice that $|(u-\psi)_+(x)-(u-\psi)_+(y)|\leq |(u-\psi)_+(x)-(u-\psi)_+(y)|\cdot (\mathbbm{1}_{\{u(x,t)>\psi(x)\}} + \mathbbm{1}_{\{u(y,t)>\psi(y)\}})$, by Hölder inequality,
    \begin{equation*}
        \begin{aligned}
            |J_2| \leq & C(\epsilon)  \int_{\mathbb{H}^n}\int_{|x^{-1} \cdot y|\leq 1}[\psi(x)-\psi(y)]^2 J(x,y)\mathbbm{1}_{\{u(x,t)>\psi(x)\}} \diff \mu(y) \diff \mu(x) \\
            & + \epsilon \mathcal{E}_{J}((u-\psi)_+ , (u-\psi)_+) 
        \end{aligned}
    \end{equation*}
    Since $\psi$ is a Lipschitz function and the equivalence (2.1), we obtain the desired inequality. 
\end{proof}

Next, we prove the first De Giorgi type oscillation lemma: if u is mostly negative in time-space cylinder, the supreme goes down in the half cylinder. Denote the cylinder $\{|x|\leq R,-a\leq t\leq 0\}$ by $\Gamma_{R,a}$.

\begin{lemma}
    There is a constant $\delta(\beta)\in (0,1)$  depending on the choice of nonlinearity $\beta$ such that, for any $0<a\leq 1$,
    if $u:\mathbb{H}^n\times[-2,0]\rightarrow \mathbb{R}$ is a weak solution to equation (7.1) satisfying
        \begin{equation}
            \begin{aligned}
                u(x,t)\leq 1+(|x|^{\alpha/4}-1)_+~~~\text{in}~\mathbb{H}^n\times[-2,0],
            \end{aligned}
        \end{equation}
    and
        \begin{equation}
            \begin{aligned}
                |\{u>0\}\cap \Gamma_{2,2a}|\leq C\delta(\beta)^{1+\frac{Q}{\alpha}}a^{(1+\frac{Q}{\alpha})},
            \end{aligned}
        \end{equation}
    then 
        \begin{equation}
            \begin{aligned}
                u(x,t)\leq 1/2 ~~\text{if}~(x,t)\in \Gamma_{1,a}.
            \end{aligned}
        \end{equation}
    In the proof, we have an explicit bound of $\delta(\beta)$,
        \begin{equation}
            \begin{aligned}
                \delta(\beta)= \frac{1}{2}\frac{\inf\limits_{1/4\leq s\leq 2}\beta'(s)}{1+\beta(2)-\beta(1/4)}.
            \end{aligned}
        \end{equation}
    % Observe that if $\inf\limits_{1/4\leq s\leq 2}\beta'(s) \rightarrow 0$, then $\delta(\beta)\rightarrow 0$. 
\end{lemma}
\begin{proof}
    We begin with establishing a nonlinear recurrence relation to the following energy quantity
    \begin{equation}
        \begin{aligned}
            U_k=\sup_{t\in[-T_k,0]}\int_{\mathbb{H}^n}(u-\psi_{L_k})_+^2(t,x)\diff \mu(x)+\int _{T_k}^0 \norm {\mathscr{L}^{\alpha/4}((u-\psi_{L_k})_+)(\cdot,t)}_2^2\diff t
         \end{aligned}
    \end{equation}
    where $L_k=\frac{1}{2}(1-\frac{1}{2^k})$, $\psi_{L_k}=L_k+(|x|^{\alpha/2}-1)_+$ and $T_k= (-1-\frac{1}{2^k}) a $. Using $u_k(t,x)$ to denote $(u-\psi_{L_k})_+(t,x)$. Since $u\geq \psi_{L_k}$, we take $M=2$ in Lemma 7. Furthermore, due to the fact that $L_k\geq 1/4$, we take $l=1/4$ in Lemma 7. For $T_{k-1}\leq \sigma\leq T_k\leq t\leq 0$, by the energy inequality (7.7) in Lemma 8, we have 
    \begin{equation}
        \begin{aligned}
            \Lambda_1\norm{(u-\psi_k)_+(\cdot,t)}_{2}^{2}+ C_1&\int _{\sigma}^{t}\norm {\mathscr{L}^{\alpha/4}((u-\psi_k)_+)(\cdot,s)}_2^2\diff s\\
            &\leq \Lambda_2\norm{(u-\psi_k)_+(\cdot,\sigma)}_1+C_2\int _{\sigma}^{t}z(s)\diff s.\nonumber
        \end{aligned}
    \end{equation}
    By taking the average of $\sigma\in [T_{k-1},T_k]$, and taking the supreme over $t\in [T_k,0]$, we obtain
    \begin{equation}
        \begin{aligned}
            U_k\leq C 2^{k}\frac{1+\Lambda_2}{\Lambda_1 \cdot a}\int_{T_{k-1}}^0z(s)\diff s.\nonumber%C_1 and C_2 can be absorbed into this C here. 2^k come from the restriction $\sigma\in[T_{k-1},T_k]$
        \end{aligned}
    \end{equation}
    Using the inequality (2.4) and Interpolation inequality, we get
    \begin{equation}
        \begin{aligned}
            \norm{u_k}_{L^{2(1+\frac{\alpha}{Q})}(\mathbb{H}^n\times[T_{k},0])}\leq C U_k^{\frac{1}{2}}.\nonumber
        \end{aligned}
    \end{equation}
    Now we deal with the controlling term $\int_{T_{k-1}}^0z(s)\diff s$. By Tchebychev inequality, we have
    \begin{equation}
        \begin{aligned}
            \int_{T_{k-1}}^0z(s)\diff s &\leq \int\int _{\mathbb{H}^n\times[T_{k-1},0]}u_{k-1}\chi_{\{u_{k-1}>\frac{1}{2^{k+1}}\}}+\int\int_{\mathbb{H}^n\times[T_{k-1},0]}\chi_{\{u_{k-1}>\frac{1}{2^{k+1}}\}} \\
            &\leq C(2^{2+\frac{2\alpha}{Q}})^k U_{k-1}^{1+\frac{\alpha}{Q}}.\nonumber
        \end{aligned}
    \end{equation}
    Therefore, we obtain
    \begin{equation}
        \begin{aligned}
            U_k\leq \frac{1+\Lambda_2}{\Lambda_1 \cdot a}(C_{Q,\Lambda,\alpha})^kU_{k-1}^{1+\frac{\alpha}{Q}}.\nonumber \\
        \end{aligned}
    \end{equation}
    If $ (\frac{1+\Lambda_2}{\Lambda_1 \cdot a})^{Q/\alpha}U_1\leq \epsilon_0(Q,\Lambda,\alpha) $, then $\lim\limits_{k\rightarrow\infty}U_k=0$ which gives  $u(x,t)\leq 1/2 ~~\text{if}~(x,t)\in \Gamma_{1,a}$. By Tchebychev inequality, we get
    \begin{equation}
        \begin{aligned}
            U_1\leq C_{Q,\Lambda,\alpha} \frac{1+\Lambda_2}{\Lambda_1 \cdot a}\int\int_{\mathbb{H}^n\times[-2a,0]}(u-(|x|^{\alpha/2}-1)_+)_+.\nonumber\\
        \end{aligned}
    \end{equation}
    To guarantee that $ (\frac{1+\Lambda_2}{\Lambda_1 \cdot a})^{Q/\alpha}U_1\leq \epsilon_0(Q,\Lambda,\alpha) $, we only need to impose the following condition,
    \begin{equation}
        \begin{aligned}
            \int\int_{\mathbb{H}^n\times[-2a,0]}(u-(|x|^{\alpha/2}-1)_+)_+\leq\epsilon_0(Q,\Lambda,\alpha)(\frac{\Lambda_1 \cdot a}{1+\Lambda_2})^{1+\frac{Q}{\alpha}}. 
        \end{aligned}
    \end{equation}
    The condition (7.13) can be connected to the condition (7.9) in the Lemma 9 through a scaling argument. For each point $(x_0,t_0)$ in $\Gamma_{1,a}$, considering the following function $u_R(x,t)$,
    \begin{equation*}
        \begin{aligned}
            u_R(x,t) = u(x_0 \cdot R^{-1}x, t_0 + R^{-\alpha}t), ~~~t\in[-2,0] ~\text{by}~\text{requiring}~R^{\alpha}>2.
        \end{aligned}
    \end{equation*}
    which is a solution to equation (7.1) corresponding to the rescaled kernel 
    \begin{equation*}
        \begin{aligned}
            J_R(x,y) = R^{-(Q+\alpha)} J(x_0 \cdot R^{-1}x, x_0^{-1} \cdot R^{-1}x)
        \end{aligned}
    \end{equation*}
    
    We claim that the rescaled solution $u_R$ satisfies (7.13). Notice that there exists large enough $R$ such that $1+((|R^{-1}x|+1)^{\alpha/4}-1)_+\leq (|x|^{\alpha/2}-1)_+$ for $|x|\geq R$. Therefore,
    \begin{equation*}
        \begin{aligned}
            \int\int_{\mathbb{H}^n\times[-2a,0]}(u_R-(|x|^{\alpha/2}-1)_+)_+ &= \int_{-2a}^0\int_{B_R(0)} {(u_R)}_+ \\
            &= R^{Q+\alpha}\int_{t_0 - 2aR^{-\alpha}}^{t_0}\int_{B_1(x_0)}u(x,t)_+ \diff \mu(x)\diff t.
        \end{aligned}
    \end{equation*} 
    Observe that $B_1(x_0)\times[t_0-2aR^{-\alpha},t_0]\subset \Gamma_{2,2a}$ if $R^{\alpha}>2$ and $(x_0,t_0)\in\Gamma_{1,a}$. Furthermore, $u(x,t)_+\leq 2^{\alpha/4}$ where $(x,t)\in \Gamma_{2,a}$. Hereafter,
    \begin{equation*}
        \begin{aligned}
            \int\int_{\mathbb{H}^n\times[-2a,0]}(u_R-(|x|^{\alpha/2}-1)_+)_+ &\leq C 2^{\alpha/4}R^{Q+\alpha}|\{u>0\}\cap\Gamma_{2,2a}| \\
            &\leq \epsilon_0(Q,\Lambda,\alpha)(\frac{\Lambda_1 \cdot a}{1+\Lambda_2})^{1+\frac{Q}{\alpha}}.
        \end{aligned}
    \end{equation*}
    by requiring
    \begin{equation*}
        \begin{aligned}
            |\{u>0\}\cap\Gamma_{2,2a}|\leq \epsilon_1(Q,\Lambda,\alpha) (\frac{\Lambda_1 \cdot a}{1+\Lambda_2})^{1+\frac{Q}{\alpha}} = c(Q,\Lambda,\alpha)\delta(\beta)^{1+\frac{Q}{\alpha}} \cdot a ^{1+\frac{Q}{\alpha}}
        \end{aligned}
    \end{equation*}
    where  $\delta(\beta)=\frac{\Lambda_1}{1+\Lambda_2}= \frac{\frac{1}{2}\inf\limits_{1/4\leq s\leq 2}\beta'(s)}{1+\beta(2)-\beta(1/4)}$ is given in the proof. Therefore, we obtain $u(x_0,t_0)\leq 1$ for each $(x_0,t_0)\in \Gamma_{1,a}$.
\end{proof}

\begin{remark}
    If u is mostly positive in the cylinder $\Gamma_{2,2^{\alpha}}$, applying  Lemma 9 to $-u$ with $\Tilde{\beta}(s)=-\beta(-s)$, we obtain  that the infimum has a upper bound in $\Gamma_{1,1}$ with the explicit bound $\delta(\Tilde{\beta})$.
\end{remark}

\begin{remark}
    We observe that the constant $\delta(\beta)$ depends on the choice of $\beta$. For instance, if $\inf\limits_{1/4\leq s\leq 2}\beta'(s)\rightarrow 0$, then $\delta(\beta)\rightarrow 0$. Hereafter, the condition (7.9) becomes harder to achieve.  
\end{remark}

The next step is to prove second De Giorgi Lemma which says some mass is lost between two level sets. Define 
\begin{equation}
    \begin{aligned}
    \psi_{\lambda}(x)=((|x|-\lambda^{-4/\alpha})_+^{\alpha/4}-1)_+, ~~\lambda\in (0,\frac{1}{3}),~x\in \mathbb{H}^n.
    \end{aligned}
\end{equation}
which can be used to control the growth at infinity.
\begin{equation}
    \begin{aligned}
    F(x)=\sup(-1,\inf(0,|x|^2-9)),~~x\in \mathbb{H}^n.
    \end{aligned}
\end{equation}
which can be used to localize the problem in the ball $B_3$. We follow the ideas in \autocite{rtfp_cl} to obtain the following lemma, 
\begin{lemma}
    Assume $C_1\leq \beta'(s)\leq C_2$ for every ${1}/{2}\leq s\leq 2$. Assume $0<a<1$. For each very small $\mu,\nu>0$, there exists $\Bar{\lambda}\in (0,\frac{1}{3})$ depending on $\mu,\nu,\alpha,a,Q,C_1,C_2$ such that for any $\lambda\in (0,\Bar{\lambda})$ and solution $u:\mathbb{H}^n\times[-2,0]\rightarrow \mathbb{R}$ to equation (7.1) satisfying
    \begin{equation}
        \begin{aligned}
            u(x,t)\leq 1+\psi_{\lambda}(x) ~~\text{in}~\mathbb{H}^n\times[-2,0] ~~\text{and}~~|\{u<0\}\cap (B_2\times(-2,-2a))|>\mu,\\
        \end{aligned}
    \end{equation}
  We have the following implication: If
  \begin{equation}
      \begin{aligned}
      |\{u>1+\lambda^2F\}\cap(B_3\times(-2a,0))|\geq \nu,
      \end{aligned}
  \end{equation}
  then
  \begin{equation}
      \begin{aligned}
      |\{1+F<u<1+\lambda^2F\}\cap(B_3\times(-2,0))|\geq \gamma(\mu,\nu,\alpha,a,Q,C_1,C_2). 
      \end{aligned}
  \end{equation}
\end{lemma}

\begin{proof}
    Define
    \begin{equation*}
        \begin{aligned}
            \phi_0 (x)= 1 + \psi_\lambda(x) + F(x),
        \end{aligned}
    \end{equation*}
    \begin{equation*}
        \begin{aligned}
            \phi_1 (x)= 1 + \psi_\lambda(x) + \lambda F(x),
        \end{aligned}
    \end{equation*}
    \begin{equation*}
        \begin{aligned}
            \phi_2 (x)= 1 + \psi_\lambda(x) + \lambda^2 F(x).
        \end{aligned}
    \end{equation*}
    Notice that the barrier function $\phi_1$ is an intermediate state between $\phi_0$ and $\phi_2$, we expect these barrier functions to provide a quantitative picture of the loss of mass between the level sets. 

    \noindent\textbf{Energy inequality:} Recall the equation (7.5) and regard the barrier function $\phi_1$ as an admissible test function, for $-2\leq t_1\leq t_2\leq 0$, we have
    \begin{equation*}
        \begin{aligned}
            \int _{\mathbb{H}^n}\mathcal{B}_{\phi_1}(u(x,t_2))\diff \mu(x) \Big|_{t_1}^{t_2}+\int _{t_1}^{t_2}\mathcal{E}_{J}(u,(u-\phi_1)_+)(t)\diff t = 0.
        \end{aligned}
    \end{equation*}
    Hereafter,
    \begin{equation}
        \begin{aligned}
            \int _{\mathbb{H}^n}\mathcal{B}_{\phi_1}(u(x,t_2))&\diff \mu(x) \Big|_{t_1}^{t_2} + \int _{t_1}^{t_2}\mathcal{E}_{J}((u-\phi_1)_+,(u-\phi_1)_+)(t)\diff t\\ 
            &+ \int _{t_1}^{t_2}\mathcal{E}_{J}((u-\phi_1)_-,(u-\phi_1)_+)(t)\diff t 
            = -\int _{t_1}^{t_2}\mathcal{E}_{J}(\phi_1,(u-\phi_1)_+)(t)\diff t.
        \end{aligned}
    \end{equation}
    In the proof of the inverse Sobolev energy inequality in Lemma 8, we neglect the nonnegative term $\int _{t_1}^{t_2}\mathcal{E}_{J}((u-\phi_1)_-,(u-\phi_1)_+)(t)\diff t $. However, controlling this term is important in the following proof. Hence, we keep this term in the above equation.  

    Now we estimate the right hand side of the above equation. Observe that $u(x) > \phi_1$ implies $x \in B_3$.  we have
    \begin{equation*}
        \begin{aligned}
            |\int _{t_1}^{t_2}\mathcal{E}_{J}(\phi_1,(u-\phi_1)_+)(t)\diff t| 
            % &\leq |2\int_{t_1}^{t_2}\int_{B_3}\int_{\mathbb{H}^n}[\phi_1(x)- \phi_1(y)][(u-\phi)_+(x)- (u-\phi_1)_+(y)]J(x,y)\diff \mu(y)\diff \mu(x)\diff t|\\
            \leq &\frac{1}{2} \int _{t_1}^{t_2}\mathcal{E}_{J}((u-\phi_1)_+,(u-\phi_1)_+)(t)\diff t \\
            & +2 \int_{t_1}^{t_2}\int_{B_3}\int_{\mathbb{H}^n} [\phi_1(x) - \phi_1(y)]^2 J(x,y) \diff \mu(y)\diff \mu(x)\diff t\\
            & := J_1 + J_2.
        \end{aligned}
    \end{equation*}
    The first term $J_1$ can be absorbed into the left hand side of the equation (7.19). Due to the definition of function $\phi_1$, we control $J_2$ in terms of $\psi_\lambda(x)$ and $F(x)$ respectively.
    \begin{equation*}
        \begin{aligned}
            J_2 & \leq 4 \lambda^2 \int_{t_1}^{t_2} \int_{\mathbb{H}^n} \int_{\mathbb{H}^n} [F(x) - F(y)]^2 J(x,y) \diff \mu(x) \diff \mu(y) \diff t \\
            &+ 4 \int_{t_1}^{t_2}\int_{B_3}\int_{\mathbb{H}^n} [\psi_\lambda(x) - \psi_\lambda(y)]^2 J(x,y) \diff \mu(x) \diff \mu(y) \diff t := L_1 + L_2.
        \end{aligned}
    \end{equation*}
    Since $F$ is a Lipschitz function with compact support, we have $L_1\leq C \lambda^2$. Furthermore, notice that $\psi_\lambda(x) = 0$ for $x\in B_3$ by requiring $0<\lambda<1/3$, we get 
    \begin{equation*}
        \begin{aligned}
            L_2  &= 4 \int_{t_1}^{t_2}\int_{B_3}\int_{\mathbb{H}^n} \psi_\lambda(y)^2 J(x,y) \diff \mu(x) \diff \mu(y) \diff t\\
            & \leq 4 \Lambda |B_3| \int_{t_1}^{t_2}\int_{\{|y| > \lambda^{-4/\alpha}\}}\frac{((|y|-\lambda^{-4/\alpha})_+^{\alpha/4}-1)_+^2}{(|y| - 3)^{Q+\alpha}}\diff \mu(y)\diff t\\
            &\leq 4 \Lambda |B_3|\lambda^2 \int_{t_1}^{t_2}\int_{\{|z|>1\}}\frac{((|z| - 1)_+^{\alpha/4}-\lambda)_+^2}{(|z|-3\lambda^{4/\alpha})^{Q+\alpha}}\diff \mu(z) \diff t\\
            &\leq C\lambda^2.
        \end{aligned}
    \end{equation*}
    Therefore, we obtained
    \begin{equation*}
        \begin{aligned}
            \int _{\mathbb{H}^n}\mathcal{B}_{\phi_1}(u(x,t_2))\diff \mu(x) \Big|_{t_1}^{t_2} &+\frac{1}{2} \int _{t_1}^{t_2}\mathcal{E}_{J}((u-\phi_1)_+,(u-\phi_1)_+)(t)\diff t\\ 
            &+ \int _{t_1}^{t_2}\mathcal{E}_{J}((u-\phi_1)_-,(u-\phi_1)_+)(t)\diff t 
            \leq C\lambda^2.
        \end{aligned}
    \end{equation*}
    Define $H(t) : =\int _{\mathbb{H}^n}\mathcal{B}_{\phi_1}(u(x,t))\diff \mu(x)$, we have
    \begin{equation*}
        \begin{aligned}
            H'(t)\leq C\lambda^2.
        \end{aligned}
    \end{equation*}
    Meanwhile, observe that $\phi_1\geq 1 - \lambda\geq 1/2$ and we choose $l = 1/2$ in Lemma 7. Also, if $u(x)> \phi_1$,  we get $x\in B_3$ and $u(x)\leq 2$. We choose $M = 2$ in Lemma 7. Furthermore, we have $(u-\phi_1)_+\leq \lambda \mathbbm{1}_{B_3}$. Hereafter, since $C_1\leq \beta'(s)\leq C_2$ for every $\frac{1}{2}\leq s\leq 2$, we obtain
    \begin{equation*}
        \begin{aligned}
            H(t) \leq C_2 \int_{\mathbb{H}^n}(u-\phi_1)_+^2\diff \mu(x) \leq C C_2\lambda^2.\\
        \end{aligned}
    \end{equation*}
    Hence, we obtain
    \begin{equation*}
        \begin{aligned}
            \int _{t_1}^{t_2}\mathcal{E}_{J}((u-\phi_1)_+,(u-\phi_1)_+)(t)\diff t \leq CC_2\lambda^2,
        \end{aligned}
    \end{equation*}
    and
    \begin{equation*}
        \begin{aligned}
            \int _{t_1}^{t_2}\mathcal{E}_{J}((u-\phi_1)_-,(u-\phi_1)_+)(t)\diff t 
            \leq CC_2\lambda^2.                             
        \end{aligned}
    \end{equation*}
    \textbf{An estimate of \boldmath{$\int_{\mathbb{H}^n} (u - \phi_1)_+^2(t,x)\diff \mu(x)$}:} 
    
    \noindent Notice that $\phi_0(x) = 0$ for $x \in B_2$. The condition (7.17) can be rewritten as $|\{u<\phi_0\}\cap (B_2\times(-2,-2a))|>\mu$. Define 
    \begin{equation*}
        \begin{aligned}
            \Sigma := \{t\in[-2,-2a]: |\{u(\cdot,t)<\phi_0\}\cap B_2|\geq \mu/4\}.
        \end{aligned}
    \end{equation*}
    Then, 
    \begin{equation*}
        \begin{aligned}
            \mu<|\{u<\phi_0\}\cap (B_2\times(-2,-2a))|\leq \mu/4 \cdot (2-2a) + |B_2|\cdot|\Sigma|,
        \end{aligned}
    \end{equation*}
    and
    \begin{equation*}
        \begin{aligned}
            |\Sigma|\geq \frac{\mu}{2|B_2|}.
        \end{aligned}
    \end{equation*}
    Next we prove that the integral $\int_{\mathbb{H}^n} (u - \phi_1)_+^2(t,x)\diff \mu(x)$ is very small for most of the time $t$ in $\Sigma$. Since $(u-\phi_1)_+\leq \lambda$, we have
    \begin{equation*}
        \begin{aligned}
            C C_2\lambda^2 &\geq \int_{-2}^{-2a} \mathcal{E}_J((u-\phi_1)_+, (u-\phi_1)_-)\diff t\\ 
            &\geq C\int_{\Sigma}\int_{B_3} \int_{\{|\{u(\cdot,t)<\phi_0\}\cap B_2|\geq \mu/4\}} (u-\phi_1)_+(x) \cdot(1-\lambda)\cdot J(x,y)\diff \mu(x)\diff \mu(y)\diff t\\
            & \geq C\Lambda^{-1}\frac{\mu}{8} \int_{\Sigma}\int_{\mathbb{H}^n} (u-\phi_1)_+(x)\diff \mu(x)\diff t\\
            & \geq C \frac{\mu}{8\lambda}\int_{\Sigma}\int_{\mathbb{H}^n}(u-\phi_1)_+^2(x) \diff \mu(x)\diff t.
        \end{aligned}
    \end{equation*}
    Hereafter, 
    \begin{equation*}
        \begin{aligned}
            \int_{\Sigma}\int_{\mathbb{H}^n}(u-\phi_1)_+^2(x) \diff \mu(x)\diff t \leq CC_2\frac{\lambda^3}{\mu}\leq \lambda^{3-\frac{1}{8}},
        \end{aligned}
    \end{equation*}
    by requiring 
    \begin{equation*}
        \begin{aligned}
            \lambda\leq (\frac{\mu}{CC_2})^8.
        \end{aligned}
    \end{equation*}
    By Tchebychev inequality, we have
    \begin{equation}
        \begin{aligned}
            \int_{\mathbb{H}^n}(u-\phi_1)_+^2(x,t)\diff \mu(x) \leq \lambda^{3-\frac{1}{4}}
        \end{aligned}
    \end{equation}
    except for a tiny set $A \subset \Sigma$ with $|A|\leq \lambda^{\frac{1}{8}}$. Furthermore, 
    \begin{equation*}
        \begin{aligned}
            |\Sigma| \geq \frac{\mu}{2|B_2|}\geq 2 \lambda^{\frac{1}{8}}~~\rm{if}~\lambda\leq (\frac{\mu}{4|B_2|})^8.
        \end{aligned}
    \end{equation*}
    Therefore, (7.20) holds for most of the time $t\in [-2,-2a]$ with a measure greater than $\mu/ 4|B_2|$.

    \noindent \textbf{Searching an intermediate set between two level sets in the setting:}

    \noindent Define 
    $\Sigma_1:= \{t\in[-2a, 0]: |\{u(\cdot,t)>\phi_2 \}|> \nu/4a\} $
    The condition (7.17) implies 
    \begin{equation*}
        \begin{aligned}
            \nu\leq |B_3|\cdot |\Sigma_1| + \frac{\nu}{4a} \cdot 2a. 
        \end{aligned}
    \end{equation*}
    Hence $|\Sigma_1|>0$ and there exists $T_0\in [-2a,0]$ such that
    \begin{equation*}
        \begin{aligned}
            |\{(u-\phi_2)_+(T_0, \cdot)>0\}|>\frac{\nu}{4a}\geq \frac{\nu}{4}.
        \end{aligned}
    \end{equation*}
    Let time $T_0$ go backwards, following the inequality (7.20) from the first step, we can find another $T_1\in[-2,-2a]$ satisfying
    \begin{equation*}
        \begin{aligned}
            \int_{\mathbb{H}^n}(u-\phi_1)_+^2(T_1,x)\diff \mu(x)\leq \lambda^{3-\frac{1}{4}}.
        \end{aligned}
    \end{equation*}
    Then,
    \begin{equation*}
        \begin{aligned}
            H(T_1)\leq C_2 \lambda^{3-\frac{1}{4}}.
        \end{aligned}
    \end{equation*}
    Moreover, we will derive the estimate of the integral $ \int_{\mathbb{H}^n}(u-\phi_1)_+^2(T_0 ,x)\diff \mu(x)$ at time $T_0$. 
    \begin{equation*}
        \begin{aligned}
            \int_{\mathbb{H}^n}(u-\phi_1)_+^2(T_0 ,x)\diff \mu(x) &\geq \int _{\mathbb{H}^n}(\phi_1-\phi_2)^2 \mathbbm{1}_{\{(u-\phi_2)_+>0\}}\\
            &\geq \int_{\mathbb{H}^n}(\lambda-\lambda^2)^2 F^2(x) \mathbbm{1}_{\{(u-\phi_2)_+>0\}}\geq C_F\frac{\lambda^2\nu^3}{4}.
        \end{aligned}
    \end{equation*}
    Then,
    \begin{equation*}
        \begin{aligned}
            H(T_0)\geq C_1 C_F\frac{\lambda^2\nu^3}{4}.
        \end{aligned}
    \end{equation*}
    By requiring $\lambda^{1-\frac{1}{4}}\leq \frac{C_1C_F}{C_2}\cdot\frac{\nu^3}{64}$, then we have $H(T_1)\leq C_1C_F\frac{\lambda^2\nu^3}{64}$. Define 
    \begin{equation*}
        \begin{aligned}
            \Sigma_2 :=\{t\in[T_1, T_0]:C_1C_F\frac{\lambda^2\nu^3}{64}<H(t)<C_1C_F\frac{\lambda^2\nu^3}{4}\}.
        \end{aligned}
    \end{equation*}
    Since $H'(t)\leq \lambda^2$, we get $|\Sigma_2|\geq C_1C_F\nu^3/8$. Now we are ready to search an intermediate level set between the level sets in (7.16) and (7.17). For any time $\tau\in \Sigma_2$, we have $|\{(u-\phi_2)_+(\tau, \cdot)\geq0\}|< \frac{\nu}{4}.$ Otherwise, we can obtain $H(\tau)\geq C_1C_F\frac{\lambda^2\nu^3}{4}$ and $\tau\notin \Sigma_2$, which is a contradiction. 

    To make sure that we can find an intermediate level set $\{\phi_0<u<\phi_2\}$ whose measure is positive, we have to give a small upper bound of $|\{(u-{\phi_0})_+(\cdot,t)\leq 0\}\cap B_3|$ for most of the time $t$ in $\Sigma_2$. Define $\Sigma_3:=\{t:|\{(u-{\phi_0})_+(\cdot,t)\leq 0\}\cap B_3|\geq \mu\}$. Then
    \begin{equation*}
        \begin{aligned}
            CC_2 \lambda^2 &\geq \int_{-2}^{0}\int_{\mathbb{H}^n}(u-\phi_1)_+(u-\phi_1)_-J(x,y)\\
            &\geq C\mu\int_{\Sigma_3}\int_{B_3}(u-\phi_1)_+\diff \mu(x)\diff t\\
            &\geq \frac{C\mu}{\lambda}|\Sigma_3|\cdot(\frac{C_1C_F\lambda^2\nu^3}{64}).
        \end{aligned}
    \end{equation*}
    Hereafter,
    \begin{equation*}
        \begin{aligned}
            |\Sigma_3|\leq\frac{CC_2\lambda}{C_1\mu\nu^3}.
        \end{aligned}
    \end{equation*}
    By requiring $\lambda<C\frac{C_1^2}{C_2}\mu\nu^6$, we have
    \begin{equation*}
        \begin{aligned}
            |\Sigma_3|\leq C_1C_F\nu^3/16 \leq |\Sigma_2|/2.
        \end{aligned}
    \end{equation*}
    Therefore, for those time $t\in \Sigma_2\backslash \Sigma_3$, we have
    \begin{equation*}
        \begin{aligned}
            |\{\phi_0<u(\cdot,t)\leq \phi_2\}\cap B_3|\geq |B_3| - \mu  -\nu/4 \geq 1/2.
        \end{aligned}
    \end{equation*}
    And
    \begin{equation}
        \begin{aligned}
            |\{1+F<u<1+\lambda^2F\}\cap(B_3\times(-2,0))|&\geq \int_{\Sigma_2\backslash\Sigma_3}1/2 \diff t\\
            &\geq |\Sigma_2|/4\geq C_1C_F\nu^3 / 8. 
        \end{aligned}
    \end{equation}
\end{proof}

\begin{remark}
    In the above proof, we use the inequality $\int_{\{|z|>1\}}\frac{((|z| - 1)_+^{\alpha/4}-\lambda)_+^2}{(|z|-3\lambda^{4/\alpha})^{Q+\alpha}}\diff \mu(z) < \infty$, which is similar to the inequality $\int_{\{\{|x|>1\}\cap \mathbb{R}^n\}}\frac{((|x| - 1)_+^{\alpha/4}-\lambda)_+^2}{(|x|-3\lambda^{4/\alpha})^{n+\alpha}}\diff x <\infty$ on $\mathbb{R}^n$. The difference is we use the homogeneous dimension $Q$ instead on the Heisenberg group $\mathbb{H}^n$.
\end{remark}

We are ready to show the oscillation Lemma by combining the first and second De Giorgi Lemma. For any $\lambda$ in Lemma 10, we define a new barrier function with more restrictive control:
\begin{equation*}
    \begin{aligned}
        H_{\lambda,\epsilon}(x)=((|x|-\lambda^{ - 4/\alpha})_+^{\epsilon}-1)_+,~~x\in \mathbb{H}^n.
    \end{aligned}
\end{equation*}
\begin{lemma}
Assume $0<C_1\leq \beta'(s)\leq C_2$ for every $1/4\leq s\leq 2$, or $-2\leq s\leq -1/4$. If $u:\mathbb{H}^n\times[-2,0]\rightarrow \mathbb{R}$ is a weak solution to equation (7.1) satisfying 
  \begin{equation}
      \begin{aligned}
      |u(x,t)|\leq 1+H_{\lambda,\epsilon}(x) ~~\text{in} ~\mathbb{H}^n\times[-2,0],
      \end{aligned}
  \end{equation}
  then there exists $0<a<1$ and $0<\theta<\frac{1}{2}$, depending on the nonlinearity function $\beta$ and the dimensional constants, such that
  \begin{equation}
      \begin{aligned}
      \sup\limits_{\Gamma_{1,a}}-\inf\limits_{\Gamma_{1,a}}\leq 2-\theta.
      \end{aligned}
  \end{equation}
\end{lemma}

\begin{proof}
    Choosing $a=1$ in Lemma 9, if 
    \begin{equation*}
        \begin{aligned}
            |\{u>0\}\cap \Gamma_{2,2}|\leq C\delta(\beta)^{1+\frac{Q}{\alpha}} 
        \end{aligned}
    \end{equation*}
    holds, then we obtain the oscillation by Lemma 9. Otherwise, if 
    \begin{equation*}
        \begin{aligned}
            |\{u>0\}\cap \Gamma_{2,2}|\geq C\delta(\beta)^{1+\frac{Q}{\alpha}} 
        \end{aligned}
    \end{equation*}
    then 
    \begin{equation*}
        \begin{aligned}
        |\{u>0\}\cap (B_2\times(-2,-2a))|> \frac{C}{2}\delta(\beta)^{1+\frac{Q}{\alpha}} 
        \end{aligned}
    \end{equation*}
    Here we take $0<a<1$ such that $|\Gamma_{2,2a}|< \frac{C}{2}\delta(\beta)^{1+\frac{Q}{\alpha}}$. Notice that $a$ only depends on the nonlinearity function $\beta$ and the dimensional constants. Therefore, we work on the solution $-u$ with the nonlinearity function:
    \begin{equation*}
        \begin{aligned}
            \Tilde{\beta}(s)=-\beta(-s)
        \end{aligned}
    \end{equation*}
    Furthermore, $-u$ satisfies (7.17) in Lemma 10 by choosing $\mu=\frac{C}{2}\delta(\beta)^{1+\frac{Q}{\alpha}}$. To be clear without misunderstandings, we assume that $u$ satisfies the condition 
    $$|\{u<0\}\cap (B_2\times(-2,-2a))|> \frac{C}{2}\delta(\beta)^{1+\frac{Q}{\alpha}}$$ and 
    $$0<C_1\leq \beta'(s)\leq C_2 ~~ \text{for}~  \text{every} ~1/4\leq s\leq 2.$$
    Considering the sequence of rescaled functions:
    \begin{equation}
        \begin{aligned}
            u_{k+1}=\frac{u_k-(1-\lambda^2)}{\lambda^2}, ~~u_0=u.\nonumber
        \end{aligned}
    \end{equation}
    Then $u_k$ is a weak solution to 
    \begin{equation}
        \begin{aligned}
            \partial_t\beta_{k}(u_k)+Lu_k=0.\nonumber
        \end{aligned}
    \end{equation}
    with a nonlinearity recurrence relation
    \begin{equation}
        \begin{aligned}
            \beta_{k+1}(s)=\frac{\beta_k(\lambda^2s+1-\lambda^2)}{\lambda^2}, ~~\beta_0=\beta. \nonumber
        \end{aligned}
    \end{equation}
    and 
    \begin{equation}
        \begin{aligned}
            \beta_k(s)=\frac{\beta(\lambda^{2k}s+1-\lambda^{2k})}{\lambda^{2k}}.\nonumber
        \end{aligned}
    \end{equation}
    Here, $\beta_k'(s)=\beta'(\lambda^{2k}s+1-\lambda^{2k})$. When $1/4\leq s\leq 2$, we have 
    $$ 2\geq \lambda^{2k}s + 1 - \lambda^{2k} \geq 1/4,~~\text{for}~\text{all}~k\geq 0. $$
    Hence, for all $k\geq 0$, we have the uniform bound of $\beta_k'(s)$,
    \begin{equation}
        \begin{aligned}
            C_1\leq \beta'(\lambda^{2k}s+1-\lambda^{2k})=\beta_k'(s)\leq C_2, ~~\text{for} ~\text{all} ~1/4\leq s\leq 2.\nonumber
        \end{aligned}
    \end{equation}
    Moreover, the constants $\delta(\beta_k)$ for $k\geq 0$ in Lemma 9 must have uniform positive lower bounds $\delta_0$. We choose $\nu$ in Lemma 10 as the constant constructed in Lemma 9: 
    $$\nu = C\delta_0^{1+\frac{Q}{\alpha}}a^{(1+\frac{Q}{\alpha})}. $$
    We choose $k_0$ as the smallest integer greater than $|B_3\times(-2,0)|/\gamma$ and $\epsilon$ small enough such that
    \begin{equation}
        \begin{aligned}
            \frac{(|x|^{\epsilon}-1)_+}{\lambda^{2(k_0+1)}}\leq (|x|^{\alpha/4}-1)_+, ~~\text{for}~\text{all}~ x\in \mathbb{H}^n.\nonumber
        \end{aligned}
    \end{equation}
    Here $\gamma$ is the constant in (7.19) corresponding to the previous determined constants $\mu, \nu, C_1, C_2$. On the other hand, since $u_{k+1}(x,t) =  1+\frac{u_k(x,t)-1}{\lambda^2}$, we have
    \begin{equation}
        \begin{aligned}
            u_k(x,t) =  1+\frac{u(x,t)-1}{\lambda^{2k}}.\nonumber
        \end{aligned}
    \end{equation}
    Therefore, for each $0\leq k \leq {k_0+1}$, 
    \begin{equation}
        \begin{aligned}
            u_k(x,t)\leq 1+ \frac{((|x|-\lambda^{-4/\alpha})_+^{\epsilon}-1)_+}{\lambda^{2(k_0+1)}}\leq 1+((|x|-\lambda^{-4/\alpha})_+^{\alpha/4}-1)_+=1+\psi_\lambda(x).\nonumber
        \end{aligned}
    \end{equation}
    As long as  $|\{u_k>1+\lambda^2F\}\cap(B_3\times(-2a,0))|\geq \nu$ for all $1\leq k\leq k_0$, by Lemma 10, we have
    \begin{equation}
        \begin{aligned}
            |\{u_{k}\geq 1 + \lambda^2 F \}\cap (B_3 \times (-2a,0))|&\leq  |\{u_{k}\geq 1 + F\}\cap (B_3 \times (-2a,0))|-\gamma\\
            & \leq|\{u_{k-1}\geq 1+\lambda^2F\}\cap (B_3 \times (-2a,0))|-\gamma\\
            &\leq |B_3\times[-2,0]| - k \cdot \gamma.\nonumber\\
        \end{aligned}
    \end{equation}
    This can not true up to $k_0$. Hence, there exists $1\leq k'\leq k_0$ such that $|\{u_{k'}>1+\lambda^2F\}\cap (B_3\times(-2a,0))|\leq \nu$. And we have,
    \begin{equation}
        \begin{aligned}
        |\{u_{k'+1}>0\}\cap (B_2\times(-2a,0))|\leq 
        |\{u_{k'}>1+\lambda^2F\}\cap(B_3\times(-2a,0))|\leq \nu.\nonumber\\
        \end{aligned}
    \end{equation}
    By Lemma 9, we have $u_{k'+1}\leq \frac{1}{2}$ in $\Gamma_{1,a}$. This gives the result with $0<\theta=\frac{\lambda^{2(k_0+2)}}{2}<\frac{1}{2}$.
\end{proof}

Now we are ready to prove the continuity of the weak solutions to equation (1.1) by contradiction. Based on the oscillation Lemma 11, we will construct the iterative sequence in which are the solutions to the rescaled equation (7.1). And the iterative sequence exhibits the quantitative behavior of the solution around some given point since they are constructed over a decreasing parabolic cylinder. 

\noindent \textbf{Proof of Theorem 3}:
\begin{proof}
    We first use translations and rescaling arguments to move each point $(x_0,t_0)$ to the origin. Let $\tau_0=\inf\{1,t_0/3\}$ and $A=\sup\limits_{t_0/3\leq t\leq t_0}\norm{u(\cdot,t)}_\infty$. The existence of $A$ is guaranteed by the smoothing effects. Define the following rescaled function:
    \begin{equation*}
        \begin{aligned}
            u_0(x,t)=\frac{1}{A}u(x_0 \cdot \tau_0^{1/\alpha}x,t_0+\tau_0t), ~~t\in[-2,0]
        \end{aligned}
    \end{equation*}
    which is a solution to equation (7.1) with the nonlinearity function $\beta_0(s)=\frac{1}{A}\beta(As)$ and the measurable kernel
    $$J_0(x,y)=\tau_0^{\frac{Q+\alpha}{\alpha}}J(x_0 \cdot \tau_0^{1/\alpha}x,x_0 \cdot \tau_0^{1/\alpha}y).$$
    From the setting, we have $|u_0(x,t)|\leq 1 $ in $\mathbb{H}^n\times[-2,0]$ and still use $u$ to denote $u_0$. \\
    Let $Q_k=\Gamma_{R^{-k},R^{-k\alpha}}$ for every $k\geq 0$ and some large enough $R>1$ to be determined later. Define the semi-oscillation of $u$ in $Q_{k-1}$,
    \begin{equation}
        \begin{aligned}
            {\omega}_k=\frac{\sup_{Q_{k-1}}u-\inf_{Q_{k-1}}u}{2}.\nonumber
        \end{aligned}
    \end{equation}
    Our goal is to prove $\omega_k\rightarrow 0$ as $k\rightarrow \infty$. We prove it by contradiction and assume $\omega_k\geq \xi>0$. Given $k\geq 1$, we define
    \begin{equation}
        \begin{aligned}
            u_k(x,t)=\frac{u(R^{-k}x,R^{-\alpha k}t)-\mu_{k}}{\omega_k},~~\mu_k=\frac{\sup_{Q_{k-1}}u+\inf_{Q_{k-1}}u}{2}.\nonumber
        \end{aligned}
    \end{equation}
    The functions $u_k$ satisfy the equation (7.1) with
    \begin{equation}
        \begin{aligned}
            \beta_k(s)=\frac{\beta(\omega_ks+\mu_k)}{\omega_k}, ~~J_k(x,y)=R^{-(Q+\alpha)k}J_0(R^{-k}x,R^{-k}y).\nonumber
        \end{aligned}
    \end{equation}
    Here, $\beta_k'(s)=\beta'(\omega_k s +\mu_k)$. When $-2\leq s\leq 2$, we have
    \begin{equation}
        \begin{aligned}
            -2&\leq\frac{-\sup_{Q_{k-1}}u+3\inf_{Q_{k-1}}u}{2}\leq \omega_k s+\mu_k\\
            &\leq \frac{2\sup_{Q_{k-1}}u-2\inf_{Q_{k-1}}u+\sup_{Q_{k-1}}u+\inf_{Q_{k-1}}u}{2}\\
            &\leq 2.\nonumber
        \end{aligned}
    \end{equation}
    Hence, $\{\beta_k'(s)\}_{k=0}^{\infty}$ have uniform positive lower bound and upper bound for $ 1/4 \leq s \leq 2$ if $\mu_k\geq 0$ or $ -2 \leq s \leq - 1/4$ if $\mu_k\leq 0$.
    
    \noindent We claim that $\{u_k\}_{k=0}^{\infty}$ satisfy condition (7.21) in Lemma 11. If $|x|\leq R$, we have $|u_k|\leq1$ by requiring $R^{\alpha}>2$. If $|x|\geq R$, choosing $R$ large enough such that $H_{\lambda,\epsilon}(R)\geq \frac{2-\xi}{\xi}$, we have
    \begin{equation}
        \begin{aligned}
            |u_k(x,t)|\leq \frac{2}{\omega_k}\leq \frac{2}{\xi}\leq 1+H_{\lambda,\epsilon}(R)\leq 1+H_{\lambda,\epsilon}(x).\nonumber\\
        \end{aligned}
    \end{equation}
    Apply Lemma 10, we have
    \begin{equation}
        \begin{aligned}
            \frac{\sup_{\Gamma_{1,a}}u_k-\inf_{\Gamma_{1,a}}u_k}{2}\leq \theta.\nonumber
        \end{aligned}
    \end{equation}
    Take $R$ large enough such that $R^{\alpha}>\frac{1}{a}$, we obtain
    \begin{equation}
        \begin{aligned}
            {\omega}_k=\frac{\sup_{Q_{k-1}}u-\inf_{Q_{k-1}}u}{2}\leq \frac{\sup_{\Gamma_{1,a}}u_{k-2}-\inf_{\Gamma_{1,a}}u_{k-2}}{2}\omega_{k-2}\leq \theta\omega_{k-2}.\nonumber 
        \end{aligned}
    \end{equation}
    which leads to a contradiction. Therefore, we proved the continuity of weak solutions to equation (1.1) for general $m>0$.
\end{proof}

We have proved the continuity of the weak solution by the oscillation Lemma 11 and the constructed iterative sequence. Based on this construction, it is natural to think about proving the Hölder continuity at some reasonable point. The critical point is to provide a uniform oscillation for each step of iteration. 

Reviewing the conditions in Lemma 11, it requires $C_1\leq \beta'(s)\leq C_2$. Let us recall that $\beta'_k(s) = \beta'(\omega_k s + \mu_k)$. Since $\omega_k\rightarrow 0$ following from the above continuity result, the easiest way is to give a good bound of $\mu_k$. Without loss of generality, if $u(x_0,t_0)>0$, then $\mu_k\geq \delta$ for some $\delta>0$ when $k$ is very large due to the continuity. Hereafter, Lemma 11 provided a uniform oscillation for the iterative sequence.

\noindent \textbf{Proof of Theorem 4:}

% \noindent \textbf{Proof of H\"older Regularity at Points $u(x_0,t_0)\neq 0$},
\begin{proof}
    By rescaling arguments, we assume $(x_0,t_0)=(0,0)$ and $u(0,0)>0$. By the continuity of $u$ just proved above, there exists a large enough $R_0$ such that 
    $$ \inf_{\Gamma_{R^{-1},R^{-\alpha}}}u \geq u(0,0)/2 ~\text{for}~ \text{each} ~R\geq R_0. $$
    Notice that $R_0$ varies for different points even with same function value. Define:
    $$u_0 = u(R_0^{-1}x, R_0^{-\alpha}t),~(x,t) \in \mathbb{H}^n \times [-2,0].$$
    Let $Q_k=\Gamma_{R^{-k},R^{-k\alpha}}$ for every $k\geq 0$ and some large enough $R>1$ to be determined later. Define the semi-oscillation of $u_0$ in $Q_{k-1}$,
    \begin{equation}
        \begin{aligned}
            {\omega}_k=\frac{\sup_{Q_{k-1}}u_0-\inf_{Q_{k-1}}u_0}{2}.\nonumber
        \end{aligned}
    \end{equation}
    \begin{equation*}
        \begin{aligned}
             \mu_k=\frac{\sup_{Q_{k-1}}u_0+\inf_{Q_{k-1}}u_0}{2}
        \end{aligned}
    \end{equation*}
    Then, by the construction of $u_0$, 
    \begin{equation*}
        \begin{aligned}
            \mu_k \geq u(0,0) / 2 ~\text{for}~k \geq 1. 
        \end{aligned}
    \end{equation*}
    We still use $u$ to denote $u_0$ and define the following iterative sequence:
    \begin{equation}
        \begin{aligned}
            u_k(x,t)=\frac{u(R^{-k}x,R^{-\alpha k}t)-\mu_{k}}{\nu_k}, ~~u_0 = u.\nonumber
        \end{aligned}
    \end{equation}
    Here, $\nu_k\geq 0$ will be determined later. Observe that $u_k$ satisfies equation (7.1) with nonlinearity $\beta_k(s)$ and measurable kernel $J_k(x,y)$,
    \begin{equation}
        \begin{aligned}
            \beta_k(s)={\frac{\beta(\nu_k s+\mu_k)}{\nu_k}}, ~~J_k(x,y)=R^{-(Q+\alpha)k}J(R^{-k}x,R^{-k}y). \nonumber
        \end{aligned}
    \end{equation}
    Since $ 2 \geq \nu_k s + \mu_k \geq \min(u(0,0)/2, 1/4)$ if $1/4 \leq s \leq 2$, the nonlinearity function $\{\beta_k'(s)\}_{k=0}^{\infty}$ share uniform upper and lower bound depending on the value $u(0,0)$. Therefore, from Lemma 11, there exists $0 < a, \theta < 1$ corresponding to the above uniform constant of $\beta_k'(s)$.
    To overcome the difficulties from $a<1$, we define  $Q_k^a = \Gamma_{R^{-k}, a R^{-k\alpha}}.$ Let
    \begin{equation*}
        \begin{aligned}
            \mu_k ' = \frac{\sup_{Q_{k-1}^a}u + \inf_{Q_{k-1}^a}u}{2},~~\omega_k ' = \frac{\sup_{Q_{k-1}^a}u - \inf_{Q_{k-1}^a}u}{2}
        \end{aligned}
    \end{equation*} 
    Notice that  $Q_k^a \subset Q_k$ and $\mu_k' \geq u(0,0) / 2 ~~\text{for}~k \geq 1$. Hence, the constants $a$ and $\theta$ in Lemma 11 don't have to be changed. To be clear, we use $\mu_k, \omega_k$ to denote $\mu_k', \omega_k'$. 
    Choose $\nu_k = \theta ^{k}$ for $k\geq 0$ and $\mu_0=0$. We claim that $u_k(x,t)\leq1+H_{\lambda,\epsilon}(x)$ and prove it by induction. Choose $R$ large enough such that 
    \begin{equation}
        \begin{aligned}
            \frac{2+H_{\lambda,\epsilon}(x/R)}{\theta}\leq 1+H_{\lambda,\epsilon}(x)~~\text{for}~|x|\geq R.\nonumber
        \end{aligned}
    \end{equation}
    Assume $u_{k}(x,t)\leq 1+H_{\lambda,\epsilon}(x)$ holds for all $0\leq k\leq k_0$, our goal is to prove $u_{{k_0}+1}(x,t)\leq 1+H_{\lambda,\epsilon}(x)$. By Lemma 11, take $R$ large enough such that $R\geq 2/a$, we obtain
    \begin{equation}
        \begin{aligned}
            {\omega}_k=\frac{\sup_{Q_{k-1}^a} u-\inf_{Q_{k-1}^a} u}{2}\leq \frac{\sup_{\Gamma_{1,a}}u_{k-1}-\inf_{\Gamma_{1,a}}u_{k-1}}{2}\nu_{k-1}\leq \theta\nu_{k-1}, ~~\text{for}~~1\leq k \leq k_0 + 1.\nonumber 
        \end{aligned}
    \end{equation}
    Therefore, we obtain
    \begin{equation*}
        \begin{aligned}
        \omega_k \leq \nu_k, ~~\text{for}~1 \leq k \leq k_0 + 1.
        \end{aligned}
    \end{equation*}
    If $|x|\leq R$ and $t\in[-2,0]$, we have 
    \begin{equation*}
        \begin{aligned}
            (R^{-(k_0+1)}x, R^{-\alpha (k_0+1)} t) \in Q_{k_0}^a
        \end{aligned}
    \end{equation*}
    Since $\omega_{k_0+1} \leq \nu_{k_0 + 1}$, we obtain
    \begin{equation*}
        \begin{aligned}
            |u_{k_0+1}(x,t)| \leq 1 ~~\text{for}~(x,t)\in B_R \times [-2,0].
        \end{aligned}
    \end{equation*}
    Furthermore, since
    \begin{equation}
        \begin{aligned}
            \frac{\mu_{k_0}-\mu_{k_0+1}}{\nu_{k_0}}&\leq \frac{\mu_{k_0}-\mu_{k_0+1}}{\omega_{k_0}}\\
            &\leq \frac{\sup_{Q_{k_0-1}^a}u+\inf_{Q_{k_0-1}^a}u-\sup_{Q_{k_0}^a}u-\inf_{Q_{k_0}^a}u}{\sup_{Q_{k_0-1}^a}u-\inf_{Q_{k_0-1}^a}u}\\\nonumber
            &\leq 1,\\
        \end{aligned}
    \end{equation}
    we have
    \begin{equation}
        \begin{aligned}
                u_{k_0+1}(x,t)&=\frac{u(R^{-(k_0+1)}x,R^{-\alpha (k_0+1)}t)-\mu_{k_0+1}}{\nu_{k_0+1}}\\
                &\leq \frac{\mu_{k_0}-\mu_{k_0+1}+\nu_{k_0}(1+H_{\lambda,\epsilon}(x/R))}{\nu_{k_0+1}}\\
                &\leq \frac{2+H_{\lambda,\epsilon}(x/R)}{\theta}\\
                &\leq 1+H_{\lambda,\epsilon}(x).\nonumber
        \end{aligned}
    \end{equation}
\end{proof}

\begin{remark}
    If the solution $u$ has a uniform positive lower bound, then we can choose $R_0 = 1$ and obtain the uniform Hölder regularity. 
\end{remark}

\begin{remark}
    As $u(x_0,t_0)$ approaches zero, the uniform bounds of $\beta_k'(s)$ approach zero. This will lead to the degeneracy of the oscillation in Lemma 11, which means we can not obtain uniform Hölder regularity from the above constructions near the degenerate or singular points ($u(x_0,t_0) = 0$).
\end{remark}

% We get a conclusion of H\"older regularity (7.25) at points satisfying $u(x_0,t_0)\neq 0$. When $u(x_0,t_0)$ approaches zero, the coeffcient $\lambda$ of $H_{\lambda,\epsilon}$ in Lemma 11 goes to zero, then the oscillation $\theta$ in Lemma 11 goes to zero and the above radius $R$ goes to infinity. Therefore, the exponent of H\"older regularity goes to zero when $u(x_0,t_0)$ approach zero. This exponent is uniform in some domain if we assume $u(x,t)\geq c_0$ with some constant $c_0$ for all $(x,t)$ in this domain. 

As mentioned in the above remark, we can not use the oscillation Lemma directly near the vanishing points ($u(x_0,t_0)=0$). However, if $m\geq 1$, we can prove the Hölder regularity for all the sign-changing solutions by constructing another iterative sequence with different nonlinearity functions. The bounds of $\beta_k'(s)$ will be well controlled by considering the two different situations near the vanishing point and we can impose the previous oscillation Lemma in a quantitative way. 

\noindent\textbf{Proof of Theorem 5:}
% \noindent \textbf{Proof of Uniform H\"older Continuity if $m\geq 1$}
\begin{proof}
    Denote 
    $$\mathcal{Q}_{R_0}^{\epsilon}((x_0,t_0))=B_{R_0^{-1}}(x_0)\times(t_0-R_0^{-{\alpha+\epsilon}},t_0),$$
    $$\mathcal{Q}_{R_0}((x_0,t_0),\omega)=B_{R_0^{-1}}(x_0)\times(t_0-R_0^{-\alpha}\omega ^{-\sigma},t_0),$$
    $$\mathcal{Q}^{\rho}_{R_0}((x_0,t_0),\omega)=B_{R_0^{-1}}(x_0)\times(t_0-\frac{1}{2}\rho R_0^{-\alpha}\omega ^{-\sigma},t_0).$$
    Here $\sigma=1-\frac{1}{m}$. There is no lack of generality to assume $(x_0,t_0)=(0,0)$ by scalings and translations. First, we consider the nonnegative solutions, the case of sign changing solution will be considered subsequently. For abrevity, denote 
    $\mathcal{Q}_{R_0}^{\epsilon}((0,0)):= \mathcal{Q}_{R_0}^{\epsilon}$, $\mathcal{Q}_R((0,0),\omega):=\mathcal{Q}_R(\omega)$ and $\mathcal{Q}^{
    \rho}_R((0,0),\omega): = \mathcal{Q}^{\rho}_R(\omega)$. Taking a fixed $R_0>0$ such that $\mathcal{Q}_{R_0}^{\epsilon}$ is contained in $\mathbb{H}^n\times[-3,0]$. Let 
    $$\omega_0 = \frac{\sup_{\mathcal{Q}_{R_0}^{\epsilon}}u - \inf_{\mathcal{Q}_{R_0}^{\epsilon}}u }{2}:=\frac{\text{osc} (u;\mathcal{Q}_{R_0}^{\epsilon})}{2}.$$ 
    There exists $R\geq R_0$ such that 
    \begin{equation}
        \omega_0^{\sigma}\geq R^{-\epsilon}.\nonumber
    \end{equation}
    Hence, 
    \begin{equation}
        \begin{aligned}
            \text{osc}(u;\mathcal{Q}_R(\omega_0))/2\leq \omega_0.\nonumber
        \end{aligned}
    \end{equation}

    \noindent (i) Assuming $\omega_0^{\sigma}\geq R_0^{-\epsilon}$, we have $\text{osc}(u;\mathcal{Q}_{R_0}(\omega_0))/2\leq \omega_0$. Define 
    \begin{equation}
        \begin{aligned}
            \nonumber\mu_0=\frac{\sup_{\mathcal{Q}_{R_0}(\omega_0)}u+\inf_{\mathcal{Q}_{R_0}(\omega_0)}u}{2}.
        \end{aligned}
    \end{equation}
    We also define
    $$u_1(x,t) = \frac{u(\frac{x}{C_0 R_0},\frac{t}{C_0^{\alpha}R_0^{\alpha}{\omega_0}^{\sigma}})-\mu_0}{\omega_0}.$$
    Choose $C_0$ large enough such that $C_0^{\alpha}\geq 2$, for $|x|\leq C_0$ and $-2 \leq t\leq 0$, we have
    \begin{equation*}
        \begin{aligned}
            (\frac{x}{C_0 R_0},\frac{t}{C_0^{\alpha}R_0^{\alpha}{\omega_0}^{\sigma}}) \in \mathcal{Q}_{R_0}(\omega_0).
        \end{aligned}
    \end{equation*}
    Then, 
    \begin{equation*}
        \begin{aligned}
            |u(x,t)|\leq 1,  ~~\text{for}~(x,t)\in B_{C_0}\times [-2,0].   
        \end{aligned}
    \end{equation*}
    Considering equation (7.1) with the following nonlinearity function,  
    \begin{equation}
        \begin{aligned}
            \beta(s)=\frac{|as+b|^{1/m}}{a^{1-\sigma}}, ~~\text{with}~ 0\leq  \frac{b}{a}\leq \frac{3}{2}.
        \end{aligned}
    \end{equation}
    Then
    \begin{equation*}
        \beta'(s)=|s+\frac{b}{a}|^{-\sigma}.
    \end{equation*}
    Notice that $\beta'(s)$ share the uniform upper and lower bound for $\frac{1}{4}\leq s\leq 2$. Let $\theta$ and $a$ be the corresponding constants in Lemma 11. Denote $\frac{2-\theta}{2}$ by $\eta$, which will be important in the following constructions. Furthermore, choose $C_0$ large enough such that 
    \begin{equation*}
        \begin{aligned}
            2 R_0^{{\epsilon}/{\sigma}} \leq 1 + H_{\lambda, \epsilon}(x),~~\text{for} ~|x|\geq C_0,
        \end{aligned}
    \end{equation*}
    and 
    \begin{equation}
        \begin{aligned}
            \frac{2+H_{\lambda,\epsilon}(\frac{x}{C_0})}{\eta}\leq 1+ H_{\lambda,\epsilon}(x),~~\text{for} ~|x|\geq C_0.
        \end{aligned}
    \end{equation}
    Then,
    \begin{equation*}
        \begin{aligned}
            |u_1(x,t)|\leq \frac{2}{\omega_0}\leq 1+H_{\lambda,\epsilon}(x)~~\text{for}~|x|\geq C_0.
        \end{aligned}
    \end{equation*}
    Here we use 
    \begin{equation}
        \begin{aligned}
            \omega_0 \geq R_0^{-\frac{\epsilon}{\sigma}} ~~\text{since}~\sigma \geq 0
        \end{aligned}
    \end{equation}
    By Lemma 11, we obtain
    \begin{equation*}
        \begin{aligned}
            \frac{\sup_{\Gamma_{1,a}}u_1 - \inf_{\Gamma_{1,a}}u_1}{2}\leq \eta.
        \end{aligned}
    \end{equation*}
    Hence, define the folllowing iterative sequence,
    \begin{equation*}
        \begin{aligned}
            R_k = K_0^{k}\cdot R_0, ~~\omega_k = \eta^{k} \cdot \omega_0,~~\text{for} ~k \geq 0.
        \end{aligned}
    \end{equation*}
    Here, $K_0$ will be determined later. Define the corresponding solution $u_k$:
    \begin{equation*}
        \begin{aligned}
            u_k(x,t) = \frac{u(\frac{x}{C_0 R_{k-1}},\frac{t}{C_0^{\alpha}R_{k-1}^{\alpha}{\omega_{k-1}^{\sigma}}})-\mu_{k-1}}{\omega_{k-1}},~~\mu_{k-1} = \frac{\sup\limits_{\mathcal{Q}_{R_{k-1}}(\omega_{k-1})}u~+~\inf\limits_{\mathcal{Q}_{R_{k-1}}(\omega_{k-1})}u}{2}.
        \end{aligned}
    \end{equation*}
    The corresponding nonlinearity functions will be:
    \begin{equation*}
        \begin{aligned}
            \beta_k(s) = \frac{|\omega_{k-1} s + \mu_{k-1}|^{1/m}}{\omega_{k-1}^{1-\sigma}},
        \end{aligned}
    \end{equation*}
    and 
    \begin{equation*}
        \begin{aligned}
            \beta_k'(s) = |s + \frac{\mu_{k-1}}{\omega_{k-1}}|^{-\sigma}. 
        \end{aligned}
    \end{equation*}
    If we assume
    $$\frac{\inf\limits_{\mathcal{Q}_{R_{k}}(\omega_{k})}u}{\omega_k}\leq \frac{1}{2}~~ \text{for}~ 0\leq k\leq k_0,$$
    we have
    \begin{equation*}
        \begin{aligned}
            \frac{\mu_{k-1}}{\omega_{k-1}}\leq \frac{3}{2}~~\text{for}~1\leq k\leq k_0+1.
        \end{aligned}
    \end{equation*}
    We claim that 
    \begin{equation*}
        \begin{aligned}
            |u_k(x,t)|\leq 1 + H_{\lambda,\epsilon}(x),~~ \text{for}~1\leq k\leq k_0+3,
        \end{aligned}
    \end{equation*}
    and 
    \begin{equation*}
        \begin{aligned}
            \text{osc}(u;\mathcal{Q}_{R_{k-1}}({\omega_{k-1}}))/2\leq \omega_{k-1},~~\text{for}~1\leq k\leq k_0+2.
        \end{aligned}
    \end{equation*}
    We prove them by induction. Assume that $|u_k(x,t)|\leq 1 + H_{\lambda,\epsilon}(x)$ and
    \begin{equation}
        \begin{aligned}
            \nonumber\text{osc}(u;\mathcal{Q}_{R_{k-1}}({\omega_{k-1}}))/2\leq \omega_{k-1}~\rm{ for} ~1\leq k\leq k_1~  \rm{with}~ 1\leq k_1\leq k_0+1.
        \end{aligned}
    \end{equation}
    Therefore, by Lemma 11, we obtain 
    \begin{equation*}
        \begin{aligned}
            \frac{\sup_{\Gamma_{1,a}}u_k - \inf_{\Gamma_{1,a}}u_k}{2} \leq \eta, ~~\text{for}~1\leq k\leq k_1.
        \end{aligned}
    \end{equation*}
    Choose $K_0$ large enough such that $K_0\geq C_0 \cdot \eta^{-\frac{\sigma}{\alpha}}\cdot a^{-\frac{1}{\alpha}}$, then
    \begin{equation*}
        \begin{aligned}
            \mathcal{Q}_{R_k}(\omega_k) \subset B_{\frac{1}{C_0 R_{k-1}}}\times [\frac{-a}{C_0^{\alpha}R_{k-1}^{\alpha}\omega_{k-1}^{\sigma}},0],~~\text{for}~1\leq k\leq k_1.
        \end{aligned}
    \end{equation*}
    Hence,
    \begin{equation*}
        \begin{aligned}
            \text{osc}(u;\mathcal{Q}_{R_k}({\omega_k}))/2\leq \omega_k, ~~\text{for}~0\leq k\leq k_1.
        \end{aligned}
    \end{equation*}
    Furthermore, if $|x|\leq C_0$ and $-2\leq t\leq 0$, we have
    \begin{equation*}
        \begin{aligned}
            (\frac{x}{C_0 R_{k_1}},\frac{t}{C_0^{\alpha}R_{k_1}^{\alpha}{\omega_{k_1}^{\sigma}}}) \subset \mathcal{Q}_{R_{k_1}}(\omega_{k_1}).
        \end{aligned}
    \end{equation*}
    Then, we obtain
    \begin{equation*}
        \begin{aligned}
            |u_{k_1+1}(x,t)|\leq 1,~~\text{for}~(x,t)\in B_{C_0}\times[-2,0].
        \end{aligned}
    \end{equation*}
    Since $\text{osc}(u;\mathcal{Q}_{R_{k_1-1}}({\omega_{k_1-1}}))/2\leq \omega_{k_1-1}$, we have
    \begin{equation*}
        \begin{aligned}
            \frac{\mu_{k_1-1}-\mu_{k_1}}{\omega_{k_1-1}}\leq 1,
        \end{aligned}
    \end{equation*}
    Therefore, by (7.24), we obtain
    \begin{equation*}
        \begin{aligned}
            u_{k_1+1}(x,t)& = \frac{\omega_{k_1-1}u_{k_1}(\frac{x}{K_0}, \frac{t}{K_0^{\alpha}\eta^{\sigma}}) + \mu_{k_1-1}-\mu_{k_1}}{\omega_{k_1}}\\
            &\leq \frac{2+H_{\lambda,\epsilon}(\frac{x}{K_0})}{\eta}\\
            &\leq \frac{2+H_{\lambda,\epsilon}(\frac{x}{C_0})}{\eta}\leq 1+H_{\lambda,\epsilon}(x), ~~\text{for}~|x|\geq C_0.
        \end{aligned}
    \end{equation*}
    Since $\text{osc}(u;\mathcal{Q}_{R_{k_1}}({\omega_{k_1}}))/2\leq \omega_{k_1}$, we also obtain
    \begin{equation*}
        \begin{aligned}
            |u_{k_1+2}(x,t)|\leq 1+H_{\lambda,\epsilon}(x,t).
        \end{aligned}
    \end{equation*}

    \noindent (ii) If there exists some $k$ such that the average $\mu_k$ is not comparatively small with the corresponding oscillation. Assume 
    $$\frac{\inf\limits_{\mathcal{Q}_{R_{k}}(\omega_{k})}u}{\omega_k}\leq \frac{1}{2}~~ \text{for}~ 0\leq k\leq k_0,$$
    and
    \begin{equation*}
        \begin{aligned}
            \frac{\inf\limits_{\mathcal{Q}_{R_{k_0+1}}(\omega_{k_0+1})} u }{\omega_{k_0+1}}\geq \frac{1}{2}.
        \end{aligned}
    \end{equation*}
    Consider 
    \begin{equation*}
        \begin{aligned}
            u_{k_0+2}(x,t) = \frac{u(\frac{x}{C_0 R_{k_0+1}},\frac{t}{C_0^{\alpha}R_{k_0+1}^{\alpha}{\omega_{k_0+1}^{\sigma}}})-\mu_{k_0+1}}{\omega_{k_0+1}},
        \end{aligned}
    \end{equation*}
    Notice that $|u_{k_0+2}(x,t)|\leq 1+H_{\lambda,\epsilon}(x,t)$ and $\inf_{\Gamma_{1,1}} u_{k_0+2}\geq \frac{1}{2}$.
    % By rescaling, we can always assume $k=0$. In this case, $\mu_0\geq2\omega_0\geq 2R_0^{-\frac{\epsilon}{\sigma}}$, which implies 
    % \begin{equation}
    %     \inf_{\mathcal{Q}_0}u \geq \omega_0\geq R_0^{-\frac{\epsilon}{\sigma}} \nonumber
    % \end{equation}
    Combining with the previous proof, we obtain the uniform H\"older regularity including the vanishing points. 

    \noindent (iii) If $\omega_0^{\sigma}\geq R_0^{-\epsilon}$ does not hold. In order to apply previous arguments, we try to see if it holds for $2R_0$ instead of $R_0$. If eventually we can find an integer $k$ such that 
    \begin{equation}
    \begin{aligned}
        \text{osc}(u;\mathcal{Q}_{2^kR_0}^{\epsilon})\geq (2^k R_0)^{-\frac{\epsilon}{\sigma}}.\nonumber
    \end{aligned}
    \end{equation}
    Then we can apply step (i) and step (ii) again. otherwise,
    \begin{equation}
    \begin{aligned}
        \text{osc}(u;\mathcal{Q}_{2^kR_0}^{\epsilon})\leq (2^k R_0)^{-\frac{\epsilon}{\sigma}}.\nonumber
    \end{aligned}
    \end{equation}
    for all $k\geq 0$. In this case, we also get an uniform H\"older regularity.

    \noindent(iv) Considering the sign-changing solutions with the initial data $u(x,0)$. Let $u(x,0): = u_+ (x,0) - u_-(x,0)$. Here, $u_+ (x,0):= \max(0, u(x,0))$ and $u_- (x,0)= -\min(0,u(x,0))$. By the $L^1$-contraction property, we obtain
    \begin{equation*}
        \begin{aligned}
            u(x,t) = u_1(x,t) - u_2(x,t).
        \end{aligned}
    \end{equation*}
    Here $u_1$ and $u_2$ are nonnegative solutions corresponding to the initial data $u_+(x,0)$ and $u_-(x,0)$. By the previous parts (i)(ii)(iii), $u_1$ and $u_2$ are uniformly Hölder continuous. Therefore, the solution $u$ is also uniformly Hölder continuous.
\end{proof}
\section*{Acknowledgements}
The author would like to thank her advisor Yannick Sire for bringing the investigations of fractional calculus and CR manifolds into her sight, suggesting studying them together, and also helpful discussions and valuable criticism on this paper.
% %\subsection{Nonlocal Nonsingular Equations}
% %By the similar proof, we can get the following conclusion of the nonlocal filtration case.
% %\begin{theorem}
% %Let $\beta\in C(\mathbb{R})\cap C^1(\mathbb{R}\backslash0)$ satisfying $\beta(0)=0$ and 
% %\begin{equation}
%     %\begin{aligned}
%         %\beta'(s)\geq m_0 ~\text{for}~\text{all} ~s\in \mathbb{R},~~\beta'(s)\leq M_0|s|^{p-1} ~\text{for} ~|s|\leq \norm{u}_\infty, ~~m_0>0,p\in(0,1)
%     %\end{aligned}
% %\end{equation}
% %Then $u$ is continuous. If $\beta$ satisfies (64) also for $l\in (0,1)$,  then $u$ is also Hölder continous at every point.
% %\end{theorem}
% %The main change is we use a different but similar energy estimates induced from condition (112) and 
% %\begin{equation}
%     %\begin{aligned}
%         %cm(u-k)_+^2\leq \int _0^{(u-k)_+}\beta'(s+k)s\diff s\leq \frac{M}{2}(u-k)_+^{l+1}    \end{aligned}
%     %\end{equation}

\printbibliography

\Addresses

\end{document}